\documentclass[11pt,showkeys]{article}

\usepackage[top=25truemm,bottom=20truemm,left=20truemm,right=20truemm]{geometry}
\usepackage{amssymb,amsthm,amsmath} 
\usepackage{graphicx}
\usepackage{latexsym}
\usepackage{comment}
\usepackage{setspace}
\usepackage{authblk}
\usepackage{lipsum}
\usepackage{mathptmx}
\usepackage{graphicx}
\usepackage{subfig}
\usepackage{url}

\theoremstyle{plain}
\newtheorem{theorem}{Theorem}[section]
\newtheorem{lemma}[theorem]{Lemma}
\newtheorem{Prop}[theorem]{Proposition}

\theoremstyle{definition}
\newtheorem{Def}[theorem]{Definition}
\newtheorem{Remark}[theorem]{Remark}

\DeclareFontFamily{U}{mathx}{}
\DeclareFontShape{U}{mathx}{m}{n}{<-> mathx10}{}
\DeclareSymbolFont{mathx}{U}{mathx}{m}{n}
\DeclareMathAccent{\widehat}{0}{mathx}{"70}
\DeclareMathAccent{\widecheck}{0}{mathx}{"71}

\makeatletter

\@addtoreset{equation}{section}
\makeatother

\begin{document}

\title{Global well-posedness for the Cauchy problem of the Zakharov-Kuznetsov equation on cylindrical spaces}
\author{Satoshi Osawa\thanks{This work was supported by JST SPRING, Grant Number JPMJSP2148.}
\quad
and
\quad
Hideo Takaoka\thanks{This work was supported by JSPS KAKENHI, Grant Number 18H01129.}
}

\newcommand{\Addresses}{{
  \bigskip
  \footnotesize

  S.~Osawa, \textsc{Department of Mathematics, Kobe University, Kobe, 657-8501, Japan}\par\nopagebreak
  \textit{E-mail address}: \texttt{sohsawa@math.kobe-u.ac.jp}

  \medskip

  H.~Takaoka, \textsc{Department of Mathematics, Kobe University, Kobe, 657-8501, Japan}\par\nopagebreak
  \textit{E-mail address}: \texttt{takaoka@math.kobe-u.ac.jp}

}}

\date{\empty}

\maketitle

\begin{abstract}
We prove that the Zakharov-Kuznetsov equation on cylindrical spaces is globally well-posed below the energy norm.
As is known, local well-posedness below energy space was obtained by the first author.
We adapt I-method to extend the solutions globally in time.
Using modified energies, we obtain the polynomial bounds on the $H^s$ growth for the global solutions.
\end{abstract}

\begin{footnotesize}
{\it $2020$ Mathematics Subject Classification Numbers.}
35Q53, 42B37.

{\it Key Words and Phrases.}
Zakharov-Kuznetsov equation, Low regularity global well-posedness, Bilinear estimates.
\end{footnotesize}

\section{Introduction}
\indent

In this paper, we study the Cauchy problem for the Zakharov-Kuznetsov equation on a cylinder:
\begin{equation}\label{ZK}
\begin{cases}
& \partial_t u +\partial_x \Delta u + u\partial_x u = 0,\\
& u(x,y,0) = u_0 (x,y),
\end{cases}
\quad (x,y) \in \mathbb{R} \times \mathbb{T}, \quad t \in \mathbb{R},
\end{equation}
where $u = u(x,y,t)$ is a real-valued function, $\mathbb{T} = \mathbb{R} / (2 \pi \mathbb{Z})$, and $\Delta = \partial^2_x + \partial^2_y$ is the Laplacian.
This equation was introduced by Zakharov and Kuznetsov \cite{Zakharov}, as a model for the propagation of ionic-acoustic waves in magnetized plasma.
This equation is one of the two-dimensional extensions of the Korteweg-de Vries (KdV) equation:
$$
\partial_t u +\partial_{xxx}  u + u\partial_x u = 0.
$$
The KdV equation is one of the famous nonlinear dispersive equation, which describes the behavior of shallow water waves.
Another two dimensional generalizations of the KdV equation is the Kadomtsev-Petviashvili (KP) equations:
$$
\partial_x\left(\partial_t u +\partial_{xxx}  u + u\partial_x u\right)\pm \partial_{yy}u = 0,
$$
where the KP-I equation corresponds to minus sign, while the KP-I\!I equation to plus sign.

The Zakharov-Kuznetsov equation \eqref{ZK} has at least two conserved quantities:
$$
E[u](t) = \frac{1}{2} \int_{\mathbb{R} \times \mathbb{T}} \left({|\nabla u (x,y,t)|}^2 - \frac{1}{3} {u(x,y,t)}^3\right)\,dxdy,
$$
$$
M[u](t)= \int_{\mathbb{R} \times \mathbb{T}} {u(x,y,t)}^2\,dxdy.
$$
These conservation laws identify $H^1(\mathbb{R}\times\mathbb{T})$ as the energy space.

For the Zakharov-Kuznetsov in the $\mathbb{R}^2$ setting, Faminskii \cite{Faminskii} proved the local well-posedness in the energy space $H^1(\mathbb{R}^2)$.
After this, Linares and Pastor \cite{Linares-Pastor} showed the local well-posedness for $s>3/4$, by proving a sharp maximal function estimates.
Gr\"{u}nrock and Herr \cite{Grunrock} proved local well-posedness in $H^s(\mathbb{R}^2)$ for $s > 1/2$, by using the Fourier restriction norm method and Strichartz estimates.
The Fourier restriction norm method utilizing the $X^{s,b}$ spaces was introduced by Bourgain \cite{Bourgain,Bourgain1} (also by Kenig, Ponce and Vega \cite{Kenig}), in order to prove the local well-posedness for nonlinear Schr\"odinger and KdV equations in low regularity Sobolev spaces.
At the same time, Molinet and Pilod \cite{Molinet} showed the local well-posedness in $H^s(\mathbb{R}^2)$ for $s > 1/2$ using bilinear Strichartz estimates.
Observing proofs of global well-posedness, Shan \cite{Shan} obtained that the solution exists globally in time to data in $H^s(\mathbb{R}^2)$ for $s>5/7$. 
Shan used the I-method based on $H^1$ conservation laws.
The I-method was introduced for global well-posedness of the KdV equation by Colliander, Keel, Staffilani, Takaoka and Tao \cite{Colliander}.
Recently, Kinoshita \cite{Kinoshita} showed the local well-posedness for $s>-1/4$,  using Loomis-Whitney inequality and an orthogonal decomposition technique.
In \cite{Kinoshita}, the local ill-posedness results for $s<-1/4$ was also obtained.
This means that the data-to-solution map from the unit ball in $H^s(\mathbb{R}^2)$ to $C([0,T]; H^s)$ fails to be smooth for any $T>0$.
As a corollary from the result of local well-posedness combining with the $L^2$ conservation law, the global well-posedness in $L^2(\mathbb{R}^2)$ follows.

In the $\mathbb{R} \times \mathbb{T}$ setting, Molinet and Pilod in the same paper as above \cite{Molinet} showed the global well-posedness in $H^s(\mathbb{R} \times \mathbb{T})$ for $s \ge 1$.
The strategy is similar to the case of $\mathbb{R}^2$ setting, however it is more difficult to make sure the proof because of the lack of the smoothing effects.
In \cite{Osawa}, we showed the local well-posedness for $s>9/10$ by refining the bilinear estimates. 

In the $\mathbb{T}^2$ setting, the local well-posedness result in $H^s(\mathbb{T}^2)$ has been obtained by Linares, Panthee, Robert and Tzvetkov \cite{Linares}.
They showed that the initial value problem is locally well-posed in $H^s(\mathbb{T}^2)$ for $s>5/3$.
The proof used short-time Strichartz estimates. 
In \cite{Schippa}, Schippa improved the local well-posedness to $s>3/2$ by using short-time bilinear Strichartz estimates.
Moreover, Kinoshita and Schippa \cite{Kinoshita2} obtained the local well-posedness for $s>1$.
They estimated the nonlinear interaction term by short-time trilinear estimates.

Observing related results on the Zakharov-Kuznetsov equation.
Yamazaki \cite{Yamazaki} showed the stability and instability results for line solitary waves of the Zakharov-Kuznetsov equations in $\mathbb{R} \times \mathbb{T}_L$, where $\mathbb{T}_L$ is the torus of length $2\pi L$.
The solitary waves are constructed by hyperbolic functions.
The long-time behavior stability of solitary waves were also studied in \cite{Bridges} and in \cite{Rousset}. 

We aim to prove that the initial value problem of the Zakharov-Kuznetsov equation in $H^s(\mathbb{R} \times \mathbb{T})$ is globally well-posed for some $s<1$. 
This proof is based on bilinear estimates developed in \cite{Molinet} and I-method \cite{Colliander}.

The main theorem of the paper is the following.
\begin{theorem}\label{thm:main}
The initial value problem of \eqref{ZK} is globally well-posed in $H^s(\mathbb{R} \times \mathbb{T})$ for $s>29/31$.
In other words, for any $u_0 \in H^s(\mathbb{R} \times \mathbb{T})$, for all $T>0$, there exists a unique solution $u$ of \eqref{ZK} such that
$$
u \in C([0,T] : H^s((\mathbb{R} \times \mathbb{T})) \cap X^{s, 1/2+}_{T}.
$$
Moreover, for all $0<T' < T$, there exists a neighborhood $\mathcal{U}$ of $u_0$ in $H^s(\mathbb{R} \times \mathbb{T})$ such that the data-to-solution map 
$$
\mathcal{U}\ni v_0 \mapsto v(t) \in C([0,T'] : H^s((\mathbb{R} \times \mathbb{T})) \cap X^{s, 1/2+}_{T'}
$$
is smooth, where $v(t)$ is a unique solution of \eqref{ZK} to the initial data $v_0$.
Here function space $X^{s,b}_T$ is defined in Section \ref{sec:pre}.
\end{theorem}

The paper is organized as follow.
In Section \ref{sec:pre}, we recall some harmonic analysis tools including Littlewood-Paley decompositions of functions.
One also reviews some preliminary results for linear estimates in the Bourgain spaces $X^{s,b}$ \cite{Bourgain,Bourgain1,Kenig}.
In Section \ref{sec:LWP}, we follow the I-method scheme \cite{Colliander}.
We give several lemmas of bilinear estimates in conjunction with rescaling argument in \cite{Colliander1} and give the local well-posedness results for rescaled data.
In Section \ref{sec:energy}, we define modified energy functional $E[Iu]$ in term of the $H^s$-norm of the solution for $s<1$.
Section \ref{sec:GWP} is devoted to the proof of main theorem.
One of the key steps in the construction of global solutions is to control the increment of the modified energy.
We estimate the growth of the modified energy and provide a priori estimates for the solutions. 

\section{Some notations of function spaces}\label{sec:pre}
\indent

In this section, we will introduce some function spaces that will be used throughout the paper.

We use the absolute value of $(\xi,q)$ as $|(\xi,q)|^2 = 3\xi^2 + q^2$ in this paper. 
For positive real number $a$ and $b$, the notion $a \lesssim b$ means that there is a constant $c$ such that $a \le cb$.
When $a \lesssim b$ and $b \lesssim a$, we write $a \sim b$.
Moreover, $b+$ means that there exists $\delta >0$ such that $b+\delta$.

Denote $\mathbb{T}_\lambda = \lambda \mathbb{T}=\mathbb{R}/(2\pi \lambda\mathbb{Z})$ for $\lambda>0$.
Let $f(x,y)$ be a function defined on $\mathbb{R} \times \mathbb{T}_{\lambda}$.
We denote the Fourier transform of $f$ with respect to the variables $(x,y)$ as $\mathcal{F}^{\lambda}_{xy}f (\xi, q)$:
$$
\mathcal{F}^{\lambda}_{xy}f(\xi,q) = \int_{\mathbb{R}}\! \int_{0}^{2 \pi\lambda} e^{-i(x\xi+ y q)} f(x,y)\,dy dx,
$$
where $(\xi,q) \in \mathbb{R}\times \mathbb{Z}/\lambda$.
Moreover, let $\mathcal{F}^{-1,\lambda}_{\xi q}$ be the Fourier inverse transform with respect to spacial variables by
$$
\mathcal{F}^{-1,\lambda}_{\xi q}f(x,y) = \frac{1}{(2\pi)^2\lambda} \sum_{q \in \mathbb{Z}/\lambda} \int_{\mathbb{R}} e^{i(\xi x+qy)} f(\xi,q) \,d \xi
$$
for $(x,y)\in\mathbb{R}\times\mathbb{T}_{\lambda}$.
For simplicity, we abbreviate $\mathcal{F}_{xy}^{\lambda}$ and $\mathcal{F}^{-1,\lambda}_{\xi q}$ by $\mathcal{F}^{\lambda}$ and $\mathcal{F}^{-1,\lambda}$, respectively, when no confusion is likely.

Let $u(x,y,t)$ be a function of $\mathbb{R} \times \mathbb{T}_{\lambda} \times \mathbb{R}$, and let $\widehat{u}^{\lambda} (\xi, q, \tau)$ be Fourier transform of $u(x, y, t)$ in a similar manner as well
$$
\widehat{u}^{\lambda} (\xi, q, \tau) = \int_{\mathbb{R}^2}\! \int_{0}^{2 \pi\lambda}  e^{-i(t\tau+ix\xi+y q)} u(x,y,t)\,dy dx dt.
$$
Similarly, define the Fourier inverse transform
$$
\widecheck{u}^{\lambda}(x,y,t) = \frac{1}{(2\pi)^3\lambda} \sum_{q \in \mathbb{Z}/\lambda} \int_{\mathbb{R}^2} e^{i (t \tau+\xi x+qy)}u(\xi,q,\tau)\, d \xi d \tau.
$$

Denote $\sigma(\xi,q)=\xi^3 + \xi q^2$.
Let $\eta \in C^{\infty}_0 (\mathbb{R})$ satisfy $0 \le \eta \le 1$, $ \eta_{[-1, 1]} =1$, and $\mathrm{supp}\,\eta \subset [-2, 2]$. 
For a dyadic number $N=2^k$ with $k \in \mathbb{N}$, we denote
$$
\phi(\xi) = \eta(\xi) - \eta(2\xi),\quad  \phi_{N} (\xi, q) = \phi(N^{-1} |(\xi, q)|),\quad
\psi_{N}(\xi, q,\tau) = \phi(N^{-1}(\tau - \sigma(\xi,q))),
$$
where $(\xi, q,\tau) \in \mathbb{R}\times \mathbb{Z}/\lambda\times\mathbb{R}$.
Here, we prepare notation of interval as $I_N=\mathrm{supp}\,\phi_N$ for a dyadic number $N$.
We define the Littlewood-Paley decomposition as
$$
P_Nu= \mathcal{F}^{-1,\lambda} (\phi_N \mathcal{F}^{\lambda}u),
\quad Q_Lu = (\psi_L \widehat{u}^{\lambda})\widecheck{~}^{\lambda}.
$$ 
We let denote for $a, b \in \mathbb{R},~a \wedge b = \min\{a,b\}$ and $a \vee b = \max\{a,b\}$.

For $s \in \mathbb{R}$, we define the Sobolev spaces $H^s_{\lambda}=H^s_{\lambda}(\mathbb{R}\times\mathbb{T}_{\lambda})$ equipped with the norm
\begin{equation}\label{eq:Plancherel}
\| f \|_{H^s_{\lambda}} = \left(\frac{1}{2\pi\lambda}\sum_{q \in \mathbb{Z}/\lambda} \int_{\mathbb{R}} \langle |(\xi, q)| \rangle^{2s} | \mathcal{F}^{\lambda} f(\xi,q)|^2 \,d\xi\right)^{1/2},
\end{equation}
where  $\langle x \rangle = 1 + |x|$.
We also use $L_{\lambda}^2=H^0_{\lambda}$.

Use the restriction operator $R_K$ with respect to the $x$ variable as
$$
R_K f (x) = \int_{\mathbb{R}} \phi (\xi K^{-1}) \mathcal{F}_x f (\xi) e^{ix\xi} \,d\xi
$$
for dyadic number $K$, where $\mathcal{F}_x$ is the Fourier transform with respect to the $x$ variable.

We note that the following properties hold \cite{Colliander1}:
$$
\int_{\mathbb{R}}\!\int_0^{2\pi\lambda}f(x,y)\overline{g(x,y)}\,dxdy=\frac{1}{\lambda}\sum_{q\in\mathbb{Z}/\lambda}\int_{\mathbb{R}}\mathcal{F}^{\lambda}f(\xi,q)\overline{\mathcal{F}^{\lambda}g(\xi,q)}\,d\xi,
$$
\begin{equation}\label{eq:convolution}
\mathcal{F}^{\lambda}(fg)(\xi,q)=\left(\mathcal{F}^{\lambda}f*^{\lambda}\mathcal{F}^{\lambda}g\right)(\xi,q)=\frac{1}{\lambda}\sum_{q_1\in\mathbb{Z}/\lambda}\int_{\mathbb{R}}\mathcal{F}^{\lambda}f(\xi-\xi_1,q-q_1)\mathcal{F}^{\lambda}g(\xi_1,q_1)\,d\xi_1.
\end{equation}

Next, we describe the Bourgain space via the Fourier transform.
Following \cite{Kenig}, we introduce a class of function spaces related to Bourgain space $X^{s,b}_{\lambda}$.
Define the Bourgain space $X^{s,b}_{\lambda}$ as follows.

\begin{Def}
For $s,b \in \mathbb{R}$ and $T>0$,
\begin{equation}\label{eq:Bourgain}
\| u \|_{X^{s,b}_{\lambda}} = \left(\frac{1}{\lambda}\sum_{q\in\mathbb{Z}/\lambda}\int_{\mathbb{R}^2}\langle \tau - \sigma(\xi ,q) \rangle^{2b}  \langle |(\xi ,q)| \rangle^{2s} |\widehat{u}(\xi ,q,\tau)|^2\,d\xi d\tau\right)^{1/2},
\end{equation}
$$
\| u \|_{X^{s,b}_{\lambda,T}} =\inf\left\{\| \overline{u} \|_{X^{s,b}_{\lambda}} \mid \overline{u} : \mathbb{R} \times \mathbb{T}_{\lambda}\rightarrow \mathbb{C},~\overline{u} |_{\mathbb{R} \times \mathbb{T}_{\lambda} \times [0,T]} = u \right\}.
$$ 
When a norm of $u$ introduced in \eqref{eq:Bourgain} is bounded, we denote $u \in X^{s,b}_{\lambda}$.
We also use the same notation $L^2_{\lambda}=X^{0,0}_{\lambda}$ for functions on $\mathbb{R}\times\mathbb{T}_{\lambda}\times\mathbb{R}$ as defined before for one on $\mathbb{R}\times\mathbb{T}_{\lambda}$, if there is no confusion. 
\end{Def}

The space $X^{s,b}_{\lambda,T}$ will be used for proof of the local well-posedness result, that is under restriction of time on $[0,T]$.

If $\lambda=1$, we denote $\mathcal{F}^{\lambda},~\mathcal{F}^{-1,\lambda},~\widehat{\cdot}^{\lambda},~ \widecheck{\cdot}^{\lambda},~H^s_{\lambda},~X_{\lambda}^{s,b},~X_{\lambda,T}^{s,b}$ by $\mathcal{F},~\mathcal{F}^{-1},~\widehat{\cdot},~ \widecheck{\cdot},~H^s,~X^{s,b},~X_{T}^{s,b}$, respectively.

\begin{Remark}
The space $X^{s,b}_{\lambda}$ will be characterized by this formula
$$
\| f \|_{X^{s,b}_{\lambda}}= \| e^{t \partial_x \Delta} f \|_{H^{b}_{t} H^{s}_{\lambda}},
$$
where $e^{-t \partial_x \Delta}$ is the operator associated linear Zakharov-Kuznetsov equation on $\mathbb{R}\times\mathbb{T}_{\lambda}$, described by
$$
\mathcal{F}^{\lambda}(e^{-t \partial_x \Delta} f) (\xi, q) = e^{it\sigma(\xi,q)} \mathcal{F}^{\lambda}f(\xi, q).
$$
\end{Remark}

We summarize some formulas which were stated in \cite{Kenig,Molinet3,Molinet, Osawa}. 
\begin{lemma}\label{lem:bourgain1}
Let $s \in \mathbb{R}$ and $b > 1/2$.
Then
$$
\| \eta(t) e^{-t \partial_x \Delta} u \|_{X^{s,b}_{\lambda}} \lesssim \| f \|_{H^s_{\lambda}},
$$
for all $f \in H^s_{\lambda}$.
\end{lemma}

\begin{lemma}\label{lem:bourgain2}
Let $s \in \mathbb{R}$ and $b>1/2$.
Then 
$$
\left \| \eta(t) \int^t_{0} e^{-(t - t') \partial_x \Delta} f(t') dt' \right \|_{X^{s,b}_{\lambda}} \lesssim \| f \|_{X^{s , b-1}_{\lambda}},
$$
for all $f \in X^{s, b-1}_{\lambda}$.
\end{lemma}

\begin{lemma}\label{lem:bourgain3}
For any $T>0$, $s \in \mathbb{R}$ and for all $-1/2 < b' \le b < 1/2$,
$$
\| f \|_{X^{s,b'}_{\lambda,T}} \lesssim T^{b-b'} \| f \|_{X^{s,b}_{\lambda,T}},
$$
for all $f \in X^{s, b}_{\lambda}$.
\end{lemma}

The polynomial $\sigma(\xi, q)$ which appeared in \eqref{eq:Bourgain} plays an important role in characterizing of solution.
Now, we define the resonance function
\begin{equation}\label{eq:res}
\mathcal{H} (\xi_1, \xi_2, q_1, q_2) = \sigma(\xi_1 + \xi_2, q_1 + q_2) - \sigma(\xi_1, q_1) - \sigma(\xi_2, q_2) = 3\xi_1 \xi_2 (\xi_1 + \xi_2) + \xi_2 q^2_1 + \xi_1 q^2_2 + 2(\xi_1 + \xi_2) q_1 q_2,
\end{equation}
which plays an important role in the control of the frequency instability range between nonlinear interactions.

\section{Rescaled solutions}\label{sec:LWP}
\indent

Throughout the paper, assume $\lambda>1$.
Recall $\mathbb{T}_\lambda = \lambda \mathbb{T}=\mathbb{R}/(2\pi \lambda\mathbb{Z})$.
By rescaling
\begin{equation}\label{eq:scaling}
u^{\lambda}(x,y,t)=\frac{1}{\lambda^2}u\left(\frac{x}{\lambda},\frac{y}{\lambda},\frac{t}{\lambda^3}\right),
\end{equation}
we consider the Cauchy problem for the Zakharov-Kuznetsov equation on $\mathbb{R}\times\mathbb{T}_{\lambda}$:
\begin{equation}\label{eq:lambdaZK}
\begin{cases}
& \partial_t u^{\lambda} +\partial_x \Delta u^{\lambda} + u^{\lambda}\partial_x u^{\lambda} = 0,\\ 
& u^{\lambda}(x,y,0) = u_0^{\lambda} (x,y)\in H^s_{\lambda},
\end{cases}
\quad (x,y) \in \mathbb{R} \times \mathbb{T}_{\lambda}, \quad t \in \mathbb{R}.
\end{equation}
If we construct the solution $u^{\lambda}(t)$ of \eqref{eq:lambdaZK} on the time interval $[0,\lambda^3 T]$, we have the solution $u(t)$ of \eqref{ZK} on $[0,T]$.
  
Let $\zeta$ be combination of spatial variables as $\zeta = (\xi, q) \in \mathbb{R} \times \mathbb{Z}/\lambda$.
Define the Fourier multiplier operator $I$, which was originally introduced in \cite{Colliander} to consider global well-posedness for KdV equation.
For $N\in 2^{\mathbb{N}}$, define the operator $I$ by
$$
\mathcal{F}^{\lambda}If(\zeta) = m(\zeta) \mathcal{F}^{\lambda}f(\zeta),$$
where $m$ is a smooth, radially symmetric, non-increasing function satisfying 
\begin{equation*}
m(\zeta) =
\begin{cases}
1, & (|\zeta| \le N), \\
\left(\frac{|\zeta|}{N}\right)^{s-1}, &  (|\zeta| \ge 2N). 
\end{cases}
\end{equation*}
On the low frequency part $|\zeta| \le N$, the operator $I$ is the identity operator, while on the high frequency, $I$ is regarded as the integral operator.
Remark that $I$ maps $H^s_{\lambda}$ functions to $H^1_{\lambda}$ one. 

In this section, we prove the following local well-posedness results on $I^{-1}H^1_{\lambda}$.

\begin{Prop}[A variant of local well-posedness]\label{prop:I-local}
Let $s>9/10$.
The Cauchy problem \eqref{eq:lambdaZK} is locally well-posed in $H^s_{\lambda}$ for data $u_0^{\lambda}$ satisfying $u_0^{\lambda}\in H^s_{\lambda}$.
Moreover, the solution exists on a time interval $[0,\delta]$ with $\delta\sim \| Iu_0^{\lambda}\|_{H^1_{\lambda}}$, and the solution $u^{\lambda}(t)$ satisfies the estimate
$$
\left\| Iu^{\lambda}\right\|_{X^{1,1/2+}_{\lambda,\delta}}\lesssim \left\|Iu_0^{\lambda}\right\|_{H^1_{\lambda}}.
$$
\end{Prop}

We give several lemmas in conjunction with rescaling argument in \cite{Colliander1}.

\begin{lemma}[Lemma 3.7 in \cite{Molinet}]\label{lem:m1}
Denote a set $\Lambda \subset \mathbb{R} \times \mathbb{Z}/\lambda$.
Let the projection on the $q$ be axis contained in a set $I \subset \mathbb{Z}/\lambda$.
Assume that there is a positive constant $C$ such that for any fixed $q_0 \in I \cap \mathbb{Z}/\lambda$ and $|\Lambda \cap \{ (\xi, q_0) \mid q_0 \in \mathbb{Z}/\lambda \} | \le C$.
Then, $|\Lambda| \le \lambda C (|I|+1)$.
\end{lemma}

Using this lemma and the mean value theorem, we have the estimates for constructing bilinear estimates.
\begin{lemma} [Lemma 3.8 in \cite{Molinet}]\label{lem:m2}
Let $I$ and $J$ be two intervals on the real line and $f:J \rightarrow \mathbb{R}$ be a smooth function. Then,
$$
\left|\{ \it{x} \in J\mid f(x) \in I \}\right| \lesssim \frac{|I|}{\rm{inf}_{\it{\xi \in J}}|\it{f}'(\xi)|}.
$$
\end{lemma}

\begin{lemma}\label{lem:m3}
Let $a \neq 0,~b,~c$ be real numbers and $I$ an interval on the real line.
Then, 
$$
\left|\{ q \in \mathbb{Z}/\lambda \mid aq^2 + bq + c \in I \} \right|\lesssim  \lambda \left(\frac{|I|^{1/2}}{|a|^{1/2}} + 1\right).
$$
\end{lemma}
\noindent

\begin{proof}
Following Lemma 3.9 in \cite{Molinet}, we only consider the case of $a>0$ and $b=c=0$.
Put $I = [as^2, at^2]$ for $s, t \in \mathbb{R}$ with $s^2 < t^2$ and $s, t \ge0$.
From the shape of parabola curve and distribution of $\mathbb{Z}/\lambda$ points, we can say
$$
\left|\{ q \in \mathbb{Z}/\lambda \mid aq^2 + bq + c \in I \}\right| \le 2\lambda(|t-s|+1).
$$
Since
$$
|t-s|=\sqrt{(t-s)^2} \le \sqrt{2(t^2 - s^2)} = \frac{\sqrt{2(at^2 - as^2)}}{\sqrt{a}} = \frac{\sqrt{2} |I|^{1/2}}{|a|^{1/2}},
$$
we obtain
$$
\left|\{ q \in \mathbb{Z}/\lambda \mid aq^2 + bq + c \in I \} \right| \lesssim  \lambda \left(\frac{|I|^{1/2}}{|a|^{1/2}} + 1\right),
$$ 
which completes the proof of lemma.
\end{proof}

Using these lemmas, we get the following bilinear estimates.
\begin{lemma}\label{lem:bilinear}
For $u,\,v\in L^2_{\lambda}$,
\begin{equation}\label{eq:b1}
\| (P_{N_1} Q_{L_1} u) (P_{N_2} Q_{L_2} v) \|_{L^2_{\lambda}}  
\lesssim (N_1 \wedge N_2) (L_1 \wedge L_2)^{1/2}  \| P_{N_1} Q_{L_1} u \|_{L^2_{\lambda}}  \|P_{N_2} Q_{L_2} v \|_{L^2_{\lambda}},
\end{equation}
\begin{equation}\label{eq:b3}
\|R_K( (P_{N_1} Q_{L_1} u) (P_{N_2} Q_{L_2} v) )\|_{L^2_{\lambda}}  \lesssim  \frac{(N_1 \wedge N_2)^{1/2}}{K^{1/4}} (L_1 \wedge L_2)^{1/2} (L_1 \vee L_2)^{1/4}  \| P_{N_1} Q_{L_1} u \|_{L^2_{\lambda}}  \|P_{N_2} Q_{L_2} v \|_{L^2_{\lambda}}.
\end{equation}
Moreover, when $N_1\wedge N_2\ll N_1\vee N_2$, for $0<\theta<1$ we have 
\begin{equation}\label{eq:b2}
\| (P_{N_1} Q_{L_1} u) (P_{N_2} Q_{L_2} v) \|_{L^2_{\lambda}}  \lesssim \frac{(N_1 \wedge N_2)^{1/2}}{N_1 \vee N_2} (L_1 \wedge L_2)^{1/2} (L_1 \vee L_2)^{1/2}  \| P_{N_1} Q_{L_1} u \|_{L^2_{\lambda}}  \|P_{N_2} Q_{L_2} v \|_{L^2_{\lambda}},
\end{equation}
\begin{equation}\label{eq:b4}
\| (P_{N_1} Q_{L_1} u) (P_{N_2} Q_{L_2} v) \|_{L^2_{\lambda}}  \lesssim \frac{(N_1 \wedge N_2)^{(1+\theta)/2}}{(N_1 \vee N_2)^{1-\theta}} (L_1 \wedge L_2)^{1/2} (L_1 \vee L_2)^{\theta/2}  \| P_{N_1} Q_{L_1} u \|_{L^2_{\lambda}}  \|P_{N_2} Q_{L_2} v \|_{L^2_{\lambda}},
\end{equation}
where $N_1, N_2, L_1, L_2, K$ are dyadic numbers.
\end{lemma}

\begin{proof}
The estimate in \eqref{eq:b4} follows from the interpolation argument with \eqref{eq:b1} and \eqref{eq:b2}, so that we prove \eqref{eq:b1}, \eqref{eq:b2} and \eqref{eq:b3} in this order.

Using the Plancherel's identity \eqref{eq:Plancherel}, convolution structure \eqref{eq:convolution} and the Cauchy-Schwarz inequality, we get
\begin{equation*}
\begin{split}
\left\| (P_{N_1} Q_{L_1} u) \right. & \left.(P_{N_2} Q_{L_2} v) \right\|_{L^2_{\lambda}}   \\
&= \left(\frac{1}{\lambda}\sum_{q\in\mathbb{Z}/\lambda}\int_{\mathbb{R}^2}\left|\left(\widehat{F}^{\lambda}(P_{N_1} Q_{L_1} u)*^{\lambda}\widehat{F}^{\lambda}(P_{N_2} Q_{L_2} v)\right)(\xi,q,\tau)\right|^2\,d\xi d\tau\right)^{1/2} \\
& \lesssim \frac{1}{\lambda^{1/2}}\sup_{(\xi, q, \tau)\in\mathbb{R}\times\mathbb{Z}/\lambda\times\mathbb{R}} |A_{\xi, q, \tau}|^{1/2} \| P_{N_1} Q_{L_1} u \|_{L^2_{\lambda}}  \|P_{N_2} Q_{L_2} v \|_{L^2_{\lambda}},
\end{split}
\end{equation*}
where
\begin{equation*}
\begin{split}
A_{\xi, q, \tau} = \{ (\xi_1, q_1, \tau_1) \in \mathbb{R} \times \mathbb{Z}/\lambda \times \mathbb{R} \mid  |(\xi_1, q_1)| \in I_{N_1},~|(\xi - \xi_1, q - q_1)| \in I_{N_2},\\
 |\tau_1 - \sigma(\xi_1, q_1)| \in I_{L_1},~|\tau - \tau_1 - \sigma(\xi - \xi_1, q - q_1)| \in I_{L_2} \}.
\end{split}
\end{equation*}
By the definition of $A_{\xi, q, \tau}$ and Lemma \ref{lem:m1} with $|I|\sim\lambda(N_1\wedge N_2)$, we obtain 
$$
|A_{\xi ,q, \tau}| \lesssim \lambda (N_1 \wedge N_2)^2 (L_1 \wedge L_2),
$$
which is \eqref{eq:b1}.

Next we prove \eqref{eq:b2}.
By the triangle inequality
\begin{equation*}
\begin{split}
|\tau_1 - \sigma(\xi_1, q_1)| + |\tau - \tau_1 - \sigma(\xi - \xi_1, q - q_1)| &\le |\tau - \sigma(\xi_1, q_1)  - \sigma(\xi - \xi_1, q - q_1)| \nonumber \\
&= |\tau - \sigma(\xi, q) - \mathcal{H} (\xi_1, \xi - \xi_1, q_1, q - q_1)|, \nonumber
\end{split}
\end{equation*}
we get
$$
|A_{\xi, q, \tau}| \lesssim (L_1 \wedge L_2) |B_{\xi, q, \tau}|,
$$
where
\begin{equation*}
\begin{split}
B_{\xi, q, \tau} = \{ (\xi_1, q_1) \in \mathbb{R} \times \mathbb{Z}/\lambda \mid |(\xi_1, q_1)| \in I_{N_1}, |(\xi - \xi_1, q - q_1)| \in I_{N_2}, \\
|\tau - \sigma(\xi, q) - \mathcal{H} (\xi_1, \xi - \xi_1, q_1, q - q_1)| \lesssim L_1 \vee L_2 \}.
\end{split}
\end{equation*}
Let us focus on the resonance function $\mathcal{H}$ in \eqref{eq:res}.
We have
\begin{equation}\label{eq:A_1}
\left| \frac{\partial \mathcal{H}}{\partial \xi_1} (\xi_1, \xi - \xi_1, q_1, q-q_1) \right| = |3 \xi_1^2 + q_1^2 - (3(\xi - \xi_1)^2 + (q-q_1)^2)| \gtrsim (N_1 \vee N_2)^2.
\end{equation}
If we define $\widetilde{B}_{\xi, q, \tau}(q_1) = \{ \xi_1 \in \mathbb{R} \mid (\xi_1, q_1) \in B_{\xi,q, \tau} \}$ for each $q_1$,  Lemma \ref{lem:m2} yields
$$
\left|\widetilde{B}_{\xi, q, \tau}(q_1)\right| \lesssim \frac{L_1 \vee L_2}{(N_1 \vee N_2)^2} .
$$
Hence
$$
|B_{\xi, q, \tau}| \lesssim \lambda \frac{(L_1 \vee L_2)(N_1 \wedge N_2)}{(N_1 \vee N_2)^2}.
$$
Combining this with $|A_{\xi, q, \tau}| \lesssim (L_1 \wedge L_2) |B_{\xi, q, \tau}|$, we obtain \eqref{eq:b2}. 

From the Cauchy-Schwarz inequality and the Plancherel's identity as before, it holds that
$$
\left\|R_K ( (P_{N_1} Q_{L_1} u)(P_{N_2} Q_{L_2} v) )\right\|_{L^2_{\lambda}}
 \lesssim \frac{1}{\lambda^{1/2}}\sup_{(\xi, q, \tau)\in\mathbb{R}\times\mathbb{Z}/\lambda\times\mathbb{R}} |A^K_{\xi, q, \tau}|^{1/2} \| P_{N_1} Q_{L_1} u \|_{L^2_{\lambda}}  \|P_{N_2} Q_{L_2} v \|_{L^2_{\lambda}},
$$
where
\begin{equation*}
\begin{split}
A^K_{\xi, q, \tau} = \{ (\xi_1, q_1, \tau_1) \in \mathbb{R} \times \mathbb{Z}/\lambda \times \mathbb{R} \mid |\xi| \sim K, ~ |(\xi_1, q_1)| \in I_{N_1},~  |(\xi - \xi_1, q - q_1)| \in I_{N_2},\\
 |\tau_1 - \sigma(\xi_1, q_1)| \in I_{L_1}, ~|\tau - \tau_1 - \sigma(\xi - \xi_1, q - q_1)| \in I_{L_2} \}.
\end{split}
\end{equation*}
Using the triangle inequality, we have
$$
|A^K_{\xi, q, \tau}| \lesssim (L_1 \wedge L_2) |B^K_{\xi, q, \tau}|,
$$
where
\begin{equation*}
\begin{split}
B^K_{\xi, q, \tau} = \{ (\xi_1, q_1) \in \mathbb{R} \times \mathbb{Z}/\lambda \mid |\xi| \sim K,~ |(\xi_1, q_1)| \in I_{N_1},~ |(\xi - \xi_1, q - q_1)| \in I_{N_2},\\
|\tau - \sigma(\xi, q) - \mathcal{H} (\xi_1, \xi - \xi_1, q_1, q - q_1)| \lesssim L_1 \vee L_2 \}.
\end{split}
\end{equation*}
For the bound of $B^K_{\xi, q, \tau}$, we calculate the second derivative of $\mathcal{H}$ as
\begin{equation}\label{eq:A_2}
\left|\frac{\partial^2 \mathcal{H}}{{\partial \xi_1}^2} (\xi_1, \xi - \xi_1, q_1, q-q_1) \right| = 6|\xi| \sim K, 
\end{equation}
for $(\xi_1, q_1) \in B^K_{\xi, q, \tau}$.
Let $\widetilde{B}^K_{\xi, q, \tau}(q_1) = \{ \xi_1 \in \mathbb{R}\mid (\xi_1, q_1) \in B^K_{\xi, q, \tau} \}$ for each $q_1$.
Combining with Lemma \ref{lem:m3}, we get
$$
|\widetilde{B}^K_{\xi, q, \tau}(q_1)| \lesssim \frac{(L_1 \vee L_2)^{1/2}}{K^{1/2}},
$$
for all $q_1 \in \mathbb{Z}/\lambda$.
Finally,
\begin{equation}\label{eq:B_1}
|B^K_{\xi, q, \tau}| \lesssim \lambda \frac{(N_1 \wedge N_2) (L_1 \vee L_2)}{K^{1/2}}
\end{equation}
and we obtain \eqref{eq:b3} by $|A^K_{\xi, q, \tau}| \lesssim (L_1 \wedge L_2) |B^K_{\xi, q, \tau}|$. 
\end{proof}

We prove the following bilinear estimates in $I^{-1}X^{1,1/2+}_{\lambda}$.
\begin{lemma}\label{prop:I-b}
For $s>9/10$,
$$
\| \partial_x I(uv)\|_{X^{1, -1/2+}_{\lambda}} \lesssim \| Iu \|_{X^{1, 1/2+}_{\lambda}} \| Iv \|_{X^{1, 1/2+}_{\lambda}}.
$$
\end{lemma}

\begin{proof}
We may assume the functions $\widehat{u}^{\lambda}$ and $\widehat{v}^{\lambda}$ are nonnegative, since by the definition of $X_{\lambda}^{s,b}$ norm.
Replacing $\langle \zeta \rangle \langle \tau-\sigma(\zeta) \rangle^{1/2+} m(\zeta) \widehat{u}(\zeta, \tau)$ and $\langle \zeta \rangle \langle \tau-\sigma(\zeta) \rangle^{1/2+} m(\zeta) \widehat{v}(\zeta, \tau)$ by $\widehat{u}(\zeta, \tau)$ and $\widehat{v}(\zeta, \tau)$, respectively, and duality argument, one has to show that
\begin{equation}\label{eq:I-b1}
\begin{split}
J= &  \frac{1}{\lambda^2}\sum_{q,q_1 \in \mathbb{Z}/\lambda} \int_{\mathbb{R}^4} \Gamma^{\xi_1, q_1, \tau_1}_{\xi, q, \tau}  \frac{m(\zeta)}{m(\zeta_1)m(\zeta-\zeta_1)} \widehat{u}^{\lambda}(\zeta_1, \tau_1) \widehat{v}^{\lambda}(\zeta-\zeta_1, \tau-\tau_1) \widehat{w}^{\lambda}(\zeta, \tau)\, d\xi d\xi_1 d \tau d \tau_1\\
\lesssim & \| u \|_{L^2_{\lambda}} \| v \|_{L^2_{\lambda}} \| w \|_{L^2_{\lambda}},
\end{split}
\end{equation}
for $w\in L^2_{\lambda}$ whose Fourier transform is nonnegative.
Here
$$
\Gamma^{\xi_1, q_1, \tau_1}_{\xi, q, \tau} =\frac{ |\xi| \langle \zeta \rangle }{ \langle \zeta_1 \rangle \langle\zeta- \zeta_1 \rangle \langle \tau-\sigma(\zeta) \rangle^{1/2-} \langle \tau_1-\sigma(\zeta_1)\rangle^{1/2+} \langle \tau-\tau_1-\sigma(\zeta-\zeta_1) \rangle^{1/2+}}.
$$
Using dyadic decomposition, we rewrite $J$ as the following 
\begin{equation*}
J = \sum_{\scriptstyle N_0,N_1,N_2 \atop\scriptstyle L_0, L_1, L_2} J^{L_0, L_1, L_2}_{N_0, N_1, N_2},
\end{equation*}
where
\begin{equation}\label{eq:JQ}
\begin{split}
J^{L_0, L_1, L_2}_{N_0, N_1, N_2} = & \frac{1}{\lambda^2}\sum_{q,q_1\in\mathbb{Z}/\lambda} \int_{\mathbb{R}^4} \Gamma^{\xi_1, q_1, \tau_1}_{\xi, q, \tau}\frac{m(\zeta)}{m(\zeta_1)m(\zeta-\zeta_1)} \\ 
& \widehat{P_{N_1} Q_{L_1} u}^{\lambda}(\zeta_1, \tau_1)  \widehat{P_{N_2} Q_{L_2} v}^{\lambda}(\zeta-\zeta_1, \tau-\tau_1)\widehat{P_{N_0} Q_{L_0} w}^{\lambda}(\xi, q, \tau)\,d\xi d\xi_1 d\tau d\tau_1.
\end{split}
\end{equation}
We decompose again $J$ to five parts in frequencies:
$$
J_{LL \rightarrow L} = \sum_{\scriptstyle N_1\vee N_2 \vee N_0 \ll N \atop{\scriptstyle L_0,L_1,L_2}} J^{L_0, L_1, L_2}_{N_0, N_1, N_2},
$$
$$
J_{LH \rightarrow H} = \sum_{\scriptstyle N_1\ll N_2\sim N_0 \atop{\scriptstyle N_0 \gtrsim N \atop{\scriptstyle L_0,L_1,L_2}}} J^{L_0, L_1, L_2}_{N_0, N_1, N_2},
$$
$$
J_{HL \rightarrow H} = \sum_{\scriptstyle N_2\ll N_1\sim N_0  \atop{\scriptstyle N_0 \gtrsim N \atop{\scriptstyle L_0,L_1,L_2}}} J^{L_0, L_1, L_2}_{N_0, N_1, N_2},
$$
$$
J_{HH \rightarrow L} = \sum_{\scriptstyle N_0 \ll N_1\sim N_2  \atop{\scriptstyle N_1 \gtrsim N \atop{\scriptstyle L_0,L_1,L_2}}}J^{L_0, L_1, L_2}_{N_0, N_1, N_2},
$$
$$
J_{HH \rightarrow H} = J - (J_{LL \rightarrow L} + J_{HL \rightarrow H} + J_{LH \rightarrow H} + J_{HH \rightarrow L}).
$$

\underline{Estimate of $J_{LL \rightarrow L}$:} 
In this case, we have
$$
\frac{m(\zeta)}{m(\zeta_1) m(\zeta-\zeta_1)}\sim 1,\quad
\Gamma^{\xi_1, q_1, \tau_1}_{\xi, q, \tau}\lesssim \frac{N_1\vee N_2}{(N_1\wedge N_2)L_0^{1/2-}L_1^{1/2+}L_2^{1/2+}}.
$$
We split this case two cases, when $N_1\wedge N_2 \ll N_1\vee N_2$, and when $N_1\wedge N_2 \sim  N_1\vee N_2$.

First, we consider the contribution of the case when $N_1\wedge N_2 \ll N_1\vee N_2$.
In this case, by \eqref{eq:b2} and the Cauchy-Schwarz inequality, we have that  the contribution of the case to $J_{LL \rightarrow L}$ is bounded by
\begin{equation*}
\begin{split}
\lesssim & \sum_{\scriptstyle N_1\vee N_2 \vee N_0 \ll N\atop{\scriptstyle N_1\wedge N_2 \ll N_1\vee N_2 \atop{\scriptstyle L_0,L_1,L_2}}}\frac{N_1\vee N_2}{(N_1\wedge N_2)L_0^{1/2-}L_1^{1/2+}L_2^{1/2+}}\| (P_{N_1} Q_{L_1} u) (P_{N_2} Q_{L_2} v)\|_{L^2_{\lambda}} \| P_{N_0} Q_{L_0} w \|_{L^2_{\lambda}}\\
\lesssim  &  \sum_{\scriptstyle \scriptstyle N_1\vee N_2 \vee N_0 \ll N\atop{\scriptstyle L_0,L_1,L_2 }}\frac{1}{(N_1\wedge N_2)^{0+}L_0^{0+}L_1^{0+}L_2^{0+}}\|P_{N_1} Q_{L_1} u\|_{L^2_{\lambda}}\|P_{N_2} Q_{L_2} v\|_{L^2_{\lambda}} \| P_{N_0} Q_{L_0} w \|_{L^2_{\lambda}}\\
\lesssim  & \| u \|_{L^2_{\lambda}} \| v \|_{L^2_{\lambda}} \| w \|_{L^2_{\lambda}},
\end{split}
\end{equation*}
where we use $N_0\lesssim N_1\vee N_2$.

Second, we consider the contribution of the case when $N_1\wedge N_2 \sim  N_1\vee N_2$, namely $N_1\sim N_2$.
Note that $P_N=P_N\widetilde{P}_N$ for $\widetilde{P}_N=P_{N/2}+P_N+P_{2N}$. 
In this case, the contribution of the case to $J_{LL \rightarrow L}$ is bounded by
\begin{equation*}
\begin{split}
\lesssim & \sum_{\scriptstyle N_1\vee N_2 \vee N_0 \ll N\atop{\scriptstyle N_1\sim N_2 \atop{\scriptstyle L_0,L_1,L_2}}}\frac{N_1\vee N_2}{(N_1\wedge N_2)L_0^{1/2-}L_1^{1/2+}L_2^{1/2+}}\|\widetilde{P}_{N_0}( (P_{N_1} Q_{L_1} u) (P_{N_2} Q_{L_2} v))\|_{L^2_{\lambda}} \| P_{N_0} Q_{L_0} w \|_{L^2_{\lambda}}\\
\lesssim  &  \sum_{\scriptstyle  N_1\sim N_2 \atop{\scriptstyle L_1,L_2 }}\frac{1}{L_1^{0+}L_2^{0+}}\|P_{N_1} Q_{L_1} u\|_{L^2_{\lambda}}\|P_{N_2} Q_{L_2} v\|_{L^2_{\lambda}} \|w \|_{L^2_{\lambda}}\\
\lesssim  & \| u \|_{L^2_{\lambda}} \| v \|_{L^2_{\lambda}} \| w \|_{L^2_{\lambda}}.
\end{split}
\end{equation*}
Hence
$$
J_{LL \rightarrow L}\lesssim  \| u \|_{L^2_{\lambda}} \| v \|_{L^2_{\lambda}} \| w \|_{L^2_{\lambda}}.
$$

\underline{Estimate of $J_{LH \rightarrow H}$:}
We split this case into two cases that $N_1\ll N \lesssim N_2 \sim N_0 $ and that $N\lesssim  N_1\ll N_2\sim N_0$. 

If $N_1\ll N \lesssim N_2 \sim N_0 $, we have
$$
\frac{m(\zeta)}{m(\zeta_1) m(\zeta-\zeta_1)}\sim 1.
$$
Applying the Cauchy-Schwarz inequality again with \eqref{eq:b2}, we have that the contribution of this case to $J^{L_0, L_1, L_2}_{N_0, N_1, N_2}$ is bounded by
\begin{equation}\label{eq:LH1}
\begin{split}
\lesssim & \frac{N_0}{N_1 L_0^{1/2-} L_1^{1/2+} L_2^{1/2+}} \| (P_{N_1} Q_{L_1} u) (P_{N_2} Q_{L_2} v)\|_{L^2_{\lambda}} \| P_{N_0} Q_{L_0} w \|_{L^2_{\lambda}}\\
\lesssim & \frac{1}{N_1^{1/2} L_0^{1/2-} L_1^{0+} L_2^{0+}} \|P_{N_1} Q_{L_1} u\|_{L^2_{\lambda}} \|P_{N_2} Q_{L_2} v\|_{L^2_{\lambda}} \| P_{N_0} Q_{L_0} w \|_{L^2_{\lambda}}\\
\lesssim & \frac{1}{N_1^{0+} L_0^{0+} L_1^{0+} L_2^{0+}} \|P_{N_1} Q_{L_1} u\|_{L^2_{\lambda}} \|P_{N_2} Q_{L_2} v\|_{L^2_{\lambda}} \| P_{N_0} Q_{L_0} w \|_{L^2_{\lambda}}.
\end{split}
\end{equation}

If $N\lesssim  N_1\ll N_2\sim N_0$, we have
$$
\frac{m(\zeta)}{m(\zeta_1) m(\zeta-\zeta_1)} \sim \frac{N_1^{1-s}}{N^{1-s}}.
$$
In a similar way to above, we have that the contribution of this case to $J^{L_0, L_1, L_2}_{N_0, N_1, N_2}$ is bounded by
\begin{equation}\label{eq:LH2}
\begin{split}
\lesssim & \frac{N_0N_1^{1-s}}{N_1N^{1-s} L_0^{1/2-} L_1^{1/2+} L_2^{1/2+}} \| (P_{N_1} Q_{L_1} u) (P_{N_2} Q_{L_2} v)\|_{L^2_{\lambda}} \| P_{N_0} Q_{L_0} w \|_{L^2_{\lambda}}\\
\lesssim & \frac{1}{N_1^{s-1/2} L_0^{1/2-} L_1^{0+} L_2^{0+}} \|P_{N_1} Q_{L_1} u\|_{L^2_{\lambda}} \|P_{N_2} Q_{L_2} v\|_{L^2_{\lambda}} \| P_{N_0} Q_{L_0} w \|_{L^2_{\lambda}}\\
\lesssim & \frac{1}{N_1^{0+} L_0^{0+} L_1^{0+} L_2^{0+}} \|P_{N_1} Q_{L_1} u\|_{L^2_{\lambda}} \|P_{N_2} Q_{L_2} v\|_{L^2_{\lambda}} \| P_{N_0} Q_{L_0} w \|_{L^2_{\lambda}}.
\end{split}
\end{equation}

Therefore, by \eqref{eq:LH1} and \eqref{eq:LH2} in conjunction with previous estimate, we have
$$
J_{LH \rightarrow H}\lesssim \| u \|_{L^2_{\lambda}} \| v \|_{L^2_{\lambda}} \| w \|_{L^2_{\lambda}},
$$
which is acceptable.

\underline{Estimate of $J_{HL \rightarrow H}$:}
The proof is same as for $J_{LH \rightarrow H}$, because of the symmetry.

\underline{Estimate of $J_{HH \rightarrow L}$:}
By symmetry, we assume $N_0\ll N_1\le N_2$.
We separate this case into two cases that $N_0 \ll  N\ll N_1 \sim N_2$ and that $N\lesssim N_0\ll N_1\sim N_2$. 

If $N_0 \ll  N\ll N_1 \sim N_2$, we have 
$$
\frac{m(\zeta)}{m(\zeta_1) m(\zeta-\zeta_1)} \sim \frac{N_1^{1-s} N_2^{1-s}}{N^{2(1-s)}}. 
$$
Applying the $L^2_{\lambda}$ norm of functions $\widetilde{(P_{N_1} Q_{L_1} u)} (P_{N_0} Q_{L_0} w)$ and $P_{N_2} Q_{L_2} v$, we have that then the contribution of this case to $J^{L_0, L_1, L_2}_{N_0, N_1, N_2}$ is bounded by
\begin{equation*}
\lesssim  \frac{N_1^{1-s} N_2^{1-s}N_0^2}{N^{2(1-s)}N_1N_2 L_0^{1/2-}L_1^{1/2+}L_2^{1/2+}} \left\|\widetilde{(P_{N_1} Q_{L_1} u)} (P_{N_0} Q_{L_0} w)\right\|_{L^2_{\lambda}} \| P_{N_2} Q_{L_2} v \|_{L^2_{\lambda}},
\end{equation*} 
where $\widehat{\widetilde{f}}^{\lambda}(\zeta, \tau) = \widehat{f}^{\lambda}(-\zeta, -\tau)$. 
We apply \eqref{eq:b4} to this and have that that the contribution of this case to $J^{L_0, L_1, L_2}_{N_0, N_1, N_2}$ is bounded by
\begin{equation}\label{eq:HH1}
\begin{split}
 \lesssim &\frac{N_1^{1-s} N_2^{1-s}N_0^2}{N^{2(1-s)}N_1N_2 L_0^{1/2-}L_1^{1/2+}L_2^{1/2+}}\frac{N_0^{(1+\theta)/2}(L_0\wedge L_1)^{1/2}(L_0\vee L_1)^{\theta/2}}{N_1^{1-\theta}} \\
 &\quad  \| P_{N_1} Q_{L_1} u \|_{L^2_{\lambda}} \| P_{N_0} Q_{L_0} w \|_{L^2_{\lambda}} \| P_{N_2} Q_{L_2} v \|_{L^2_{\lambda}} \\
\lesssim &\frac{1}{N^{1-s}N_1^{s-1/2-3\theta/2}L_0^{0+}L_1^{0+}L_2^{0+}}  \| P_{N_1} Q_{L_1} u \|_{L^2_{\lambda}} \| P_{N_0} Q_{L_0} w \|_{L^2_{\lambda}} \| P_{N_2} Q_{L_2} v \|_{L^2_{\lambda}},
\end{split}
\end{equation}
for small $\theta>0$.

If $N\lesssim N_0\ll N_1\sim N_2$, we have 
$$
\frac{m(\zeta)}{m(\zeta_1) m(\zeta-\zeta_1)} \sim \frac{N_1^{1-s}N_2^{1-s}}{N^{1-s} N_0^{1-s}}.
$$
Similarly, the contribution of this case to $J^{L_0, L_1, L_2}_{N_0, N_1, N_2}$ is bounded by
\begin{equation}\label{eq:HH2}
\lesssim \frac{1}{N_1^{s-1/2-3\theta/2}N^{1-s}L_0^{0+}L_1^{0+}L_2^{0+}}  \| P_{N_1} Q_{L_1} u \|_{L^2_{\lambda}} \| P_{N_0} Q_{L_0} w \|_{L^2_{\lambda}} \| P_{N_2} Q_{L_2} v \|_{L^2_{\lambda}},
\end{equation}
for small $\theta>0$.

Summing up with respect to $N_0, N_1, N_2, L_0, L_1, L_2$, we obtain
$$
J_{HH \rightarrow L} \lesssim \| u \|_{L^2_{\lambda}} \| v \|_{L^2_{\lambda}} \| w \|_{L^2_{\lambda}}.
$$

\underline{Estimate of $J_{HH \rightarrow H}$:}
In this case, we may assume $N\ll N_0\sim N_1\sim N_2$ and have
$$
\frac{m(\zeta)}{m(\zeta_1) m(\zeta-\zeta_1)} \sim \frac{N_0^{1-s}}{N^{1-s}},\quad
\Gamma^{\xi_1, q_1, \tau_1}_{\xi, q, \tau}\sim \frac{|\xi|}{N_0 L_0^{1/2-}L_1^{1/2+}L_2^{1/2+}}.
$$
Now, we separate this case into five subcases:
\begin{itemize}
\item[(i)]
$|\xi|\lesssim 1$,
\item[(ii)]
$|\xi_1|\wedge |\xi-\xi_1|\lesssim 1 \ll |\xi|$,
\item[(iii)]
$1\ll |\xi_1|\wedge |\xi-\xi_1|\wedge |\xi|$ and  $||\zeta|^2-|\zeta_1|^2|\vee ||\zeta_1|^2-|\zeta-\zeta_1|^2|\vee ||\zeta-\zeta_1|^2-|\zeta|^2|\gtrsim N_0^{6/5}(L_0\vee L_1\vee L_1)^{0+}$,
\item[(iv)]
$1\ll |\xi_1|\wedge |\xi-\xi_1|\wedge |\xi|,~||\zeta|^2-|\zeta_1|^2|\wedge ||\zeta_1|^2-|\zeta-\zeta_1|^2|\wedge ||\zeta-\zeta_1|^2-|\zeta|^2|\ll N_0^{6/5}(L_0\vee L_1\vee L_1)^{0+}$ and $(|\xi|\vee |\xi_1|\vee |\xi-\xi_1|)(|\xi|\wedge |\xi_1|\wedge |\xi-\xi_1|)\gg N_0^{6/5}(L_0\vee L_1\vee L_1)^{0+}$,
\item[(v)]
$1\ll |\xi_1|\wedge |\xi-\xi_1|\wedge |\xi|,~||\zeta|^2-|\zeta_1|^2|\wedge ||\zeta_1|^2-|\zeta-\zeta_1|^2|\wedge ||\zeta-\zeta_1|^2-|\zeta|^2|\ll N_0^{6/5}(L_0\vee L_1\vee L_1)^{0+}$ and $(|\xi|\vee |\xi_1|\vee |\xi-\xi_1|)(|\xi|\wedge |\xi_1|\wedge |\xi-\xi_1|)\lesssim N_0^{6/5}(L_0\vee L_1\vee L_1)^{0+}$.
\end{itemize}

\underline{Subcase (i):}
Denote
$$
J_{N_0, N_1, N_2}^{L_0,L_1,L_2} =\sum_{k\in\mathbb{N}}J^{L_0, L_1, L_2}_{N_0, N_1, N_2}(k),
$$
where
\begin{equation*}
\begin{split}
J^{L_0, L_1, L_2}_{N_0, N_1, N_2}(k)  =   \frac{1}{\lambda^2}\sum_{q,q_1\in\mathbb{Z}/\lambda} \int_{Y_k} & \Gamma^{\xi_1, q_1, \tau_1}_{\xi, q, \tau}   \frac{m(\zeta)}{m(\zeta_1)m(\zeta-\zeta_1)} \\
&  \widehat{P_{N_1} Q_{L_1} u}^{\lambda}(\zeta_1, \tau_1)\widehat{P_{N_2} Q_{L_2} v}^{\lambda}(\zeta-\zeta_1, \tau-\tau_1)\widehat{P_{N_0} Q_{L_0} w}^{\lambda}(\zeta, \tau)\, d\xi d\xi_1 d\tau d\tau_1,
\end{split}
\end{equation*}
$$
Y_k = \{ (\xi, \xi_1, \tau, \tau_1) \in \mathbb{R}^4 \mid |\xi|\sim 2^{-k}\}.
$$ 
We apply the Cauchy-Schwarz inequality, and we have that the contribution of this case to $J^{L_0, L_1, L_2}_{N_0, N_1, N_2}(k)$ is bounded by
\begin{equation*}
\lesssim \frac{2^{-k}}{N^{1-s} N_0^s L_0^{1/2-} L_1^{1/2+} L_2^{1/2+}}\left\| R_{2^{-k}}\left((P_{N_1} Q_{L_1} u) (P_{N_2} Q_{L_2} v)\right)\right\|_{L^2_{\lambda}} \| P_{N_0} Q_{L_0} w \|_{L^2_{\lambda}}.
\end{equation*}
For  $|\xi| \sim 2^{-k}$, we obtain by \eqref{eq:res}
$$
\left |\frac {\partial^2 \mathcal{H}} {\partial \xi^2_1} (\xi_1, \xi - \xi_1, q_1, q - q_1) \right | = 6|\xi| \sim 2^{-k} .
$$ 
Using \eqref{eq:b3}, we get
$$
\|R_{2^{-k}} ( (P_{N_1} Q_{L_1} u) (P_{N_2} Q_{L_2} v ))\|_{L^2_{\lambda}} \lesssim 2^{k/4} N_0^{1/2} (L_1 \vee L_2)^{1/4} (L_1 \wedge L_2)^{1/2} \| P_{N_1} Q_{L_1} u \|_{L^2_{\lambda}} \| P_{N_2} Q_{L_2} v \|_{L^2_{\lambda}}.
$$
Then
\begin{equation}\label{eq:v-1}
J^{L_0, L_1, L_2}_{N_0, N_1, N_2}(k)  \lesssim  \frac{2^{-3k/4}}{N^{1-s}N_0^{s-1/2}L_0^{0+}L_1^{0+}L_2^{0+}} \| P_{N_1} Q_{L_1} u \|_{L^2_{\lambda}} \| P_{N_2} Q_{L_2} v \|_{L^2_{\lambda}} \| P_{N_0} Q_{L_0} w \|_{L^2_{\lambda}}.
\end{equation}
Summing with respect to dyadic numbers $N_0, N_1, N_2, L_0, L_1, L_2$ and $k\in\mathbb{N}$, we have the desired bound for the contribution of this case to $J_{HH \rightarrow H}$.

\underline{Subcase (ii):}
By symmetry, we may suppose $|\xi-\xi_1| \le |\xi_1|\wedge 1$.
First, we calculate the resonance function \eqref{eq:res} as
$$
\frac{\partial \mathcal{H}}{\partial \xi_1}(\xi_1, \xi-\xi_1, q, q-q_1) = 3 \xi (\xi-2\xi_1) +q(q-2q_1). 
$$

First we consider the case when $|\partial \mathcal{H}/\partial \xi_1| \gtrsim \xi^2$.
We shall use the dyadic decomposition $|\xi|\sim K$.
By the Cauchy-Schwarz inequality such as Case (i) above, we have that the contribution of this case to $J^{L_0, L_1, L_2}_{N_0, N_1, N_2}$ is bounded by
$$
\lesssim \sum_{K\in 2^{\mathbb{N}}} \frac{K}{N^{1-s}N_0^s L_0^{1/2-}L_1^{1/2+}L_2^{1/2+}}\|R_{K} ( (P_{N_1} Q_{L_1} u) (P_{N_2} Q_{L_2} v))\|_{L^2_{\lambda}} \|R_{K} P_{N_0} Q_{L_0} w \|_{L^2_{\lambda}}.
$$
Recall the proof of Lemma \ref{lem:bilinear}, we have
\begin{equation*}
\|R_{K} (  (P_{N_1} Q_{L_1} u) (P_{N_2} Q_{L_2} v) )\|_{L^2_{\lambda}} \lesssim \frac{1}{\lambda^{1/2}}\sup_{(\xi,q, \tau)\in\mathbb{R}\times\mathbb{Z}/\lambda\times\mathbb{R}} |A_{\xi,q,\tau}(K)|^{1/2} \| P_{N_1} Q_{L_1} u \|_{L^2_{\lambda}}  \|P_{N_2} Q_{L_2} v \|_{L^2_{\lambda}},
\end{equation*}
where
\begin{equation*}
\begin{split}
A_{\xi, q, \tau}(K) = \{ (\xi_1, q_1, \tau_1) \in \mathbb{R} \times \mathbb{Z}/\lambda \times \mathbb{R} \mid &  |\xi| \sim K,~ |\zeta_1| \in I_{N_1}, ~|\zeta-\zeta_1| \in I_{N_2},\\
& |\tau_1 - \sigma(\zeta_1)| \in I_{L_1}, |\tau - \tau_1 - \sigma(\zeta-\zeta_1)| \in I_{L_2} \}.
\end{split}
\end{equation*}
From the triangle inequality,
$$
|A_{\xi, q, \tau}(K)| \lesssim (L_1 \wedge L_2) |B_{\xi, q, \tau}(K)|,
$$
where
\begin{equation*}
\begin{split}
B_{\xi, q, \tau}(K) = \{ (\xi_1, q_1) \in \mathbb{R} \times \mathbb{Z}/\lambda \mid & |\xi| \sim K,~ |\zeta_1| \in I_{N_1}, |\zeta-\zeta_1| \in I_{N_2},\\
& |\tau - \sigma(\zeta) - \mathcal{H} (\xi_1, \xi - \xi_1, q_1, q - q_1)| \lesssim L_1 \vee L_2 \}.
\end{split}
\end{equation*}
Let $\widetilde{B}_{\xi, q, \tau}(K,q_1) = \{ \xi_1 \in \mathbb{R}\mid  \zeta_1 \in B_{\xi, q, \tau}(K) \}$ for each $q_1$. 
Then combining Lemma \ref{lem:m2} and the bound $|\partial \mathcal{H}/\partial \xi_1| \gtrsim \xi^2 \sim K^2$, we get
$$
|\widetilde{B}_{\xi, q, \tau}(K,q_1)| \lesssim \frac{L_1 \vee L_2}{K^2}.
$$
Hence,  we obtain
$$
|B_{\xi, q, \tau}| \lesssim \lambda \frac{(L_1 \vee L_2)(N_1 \wedge N_2)}{K^2}.
$$
Then, 
$$
\|R_{K}  ((P_{N_1} Q_{L_1} u) (P_{N_2} Q_{L_2} v))\|_{L^2_{\lambda}} \lesssim \frac{N_0^{1/2} (L_1 \vee L_0)^{1/2} (L_1 \wedge L_0)^{1/2}}{K}\| P_{N_1} Q_{L_1} u \|_{L^2_{\lambda}} \| P_{N_0} Q_{L_0} w \|_{L^2_{\lambda}}.
$$
Combining the Cauchy-Schwarz inequality and bilinear estimate, we obtain that the contribution of this case to $J^{L_0, L_1, L_2}_{N_0, N_1, N_2}$ is bounded by
\begin{equation}\label{eq:v-2-1}
\begin{split}
\lesssim &  \sum_{K\lesssim N_0} \frac{1}{N^{1-s}N_0^{s-1/2}L_0^{0+} L_1^{0+} L_2^{0+}}\| P_{N_1} Q_{L_1} u \|_{L^2_{\lambda}} \| P_{N_2} Q_{L_2} v \|_{L^2_{\lambda}} \|R_K P_{N_0} Q_{L_0} w \|_{L^2_{\lambda}}  \\
\lesssim & \frac{1}{N^{1-s}N_0^{s-1/2-}L_0^{0+} L_1^{0+} L_2^{0+}}\| P_{N_1} Q_{L_1} u \|_{L^2_{\lambda}} \| P_{N_2} Q_{L_2} v \|_{L^2_{\lambda}} \|P_{N_0} Q_{L_0} w \|_{L^2_{\lambda}}.
\end{split}
\end{equation}
Again summing with respect to dyadic numbers $N_0, N_1, N_2, L_0, L_1, L_2$, we have the desired bound for the contribution of this case to $J_{HH \rightarrow H}$.

Next, we consider the case when $|\partial \mathcal{H}/\partial \xi_1| \ll \xi^2$.
In this case, we have
\begin{equation*}
\begin{split}
\frac{\partial \mathcal{H}}{\partial \xi_1}(\xi_1, \xi - \xi_1, q_1, q-q_1) =& 3 \xi (\xi-\xi_1) -3\xi \xi_1 + q(q-q_1) -qq_1\\
=& 6\xi (\xi - \xi_1) -3 \xi^2 + q(q - q_1) - qq_1\\
=& O(\xi) - 3 \xi^2 + q(q-q_1) -qq_1.
\end{split}
\end{equation*}
To satisfy the hypothesis of resonance function $|\partial \mathcal{H}/\partial \xi_1| \ll \xi^2,~N_0 \sim N_1 \sim N_2$ and $|\xi - \xi_1| \lesssim 1\ll |\xi|$, we need $|q-q_1|\sim N_0,~|\xi|\lesssim |q|\wedge |q_1|$ to set $q(q-q_1) -qq_1\sim \xi^2$.
Furthermore we have
$$
|q| \sim |q_1| \sim |q-q_1| \sim |\xi| \sim N_0,\quad 3\xi \xi_1 (\xi - \xi_1) \sim O(\xi^2),\quad (\xi-\xi_1)(q^2 - (q-q_1)^2) \sim O(\xi^2).
$$
Let us recall the resonance function $\mathcal{H}$ in \eqref{eq:res}
\begin{equation*}
\mathcal{H}(\xi_1, \xi - \xi_1, q_1, q - q_1)=3\xi \xi_1 (\xi - \xi_1) + (\xi - \xi_1)(q^2-(q-q_1)^2) + (q-q_1)\xi_1 (q+q_1),
\end{equation*}

If $|\mathcal{H}(\xi_1, \xi - \xi_1, q_1, q - q_1)| \gtrsim \xi^2$, then $L_0\vee L_1\vee L_2\gtrsim \xi^2$ by \eqref{eq:res}.
Hence we have $|\xi|\lesssim N_0^{0+}(L_0\vee L_1\vee L_2)^{1/2-}$.
The same estimate as in \eqref{eq:LH2} implies that the contribution of this case to $J_{HH \rightarrow H}$ is bounded by
\begin{equation}\label{eq:v-2-2}
\begin{split}
\lesssim & \frac{N_0^{0+}N_0^{1-s}}{N_0N^{1-s} L_0^{0+} L_1^{0+} L_2^{0+}} \| (P_{N_1} Q_{L_1} u) (P_{N_2} Q_{L_2} v)\|_{L^2_{\lambda}} \| P_{N_0} Q_{L_0} w \|_{L^2_{\lambda}}\\
\sim & \frac{1}{N^{1-s}N_0^{s-} L_0^{0+} L_1^{0+} L_2^{0+}} \|P_{N_1} Q_{L_1} u\|_{L^2_{\lambda}} \|P_{N_2} Q_{L_2} v\|_{L^2_{\lambda}} \| P_{N_0} Q_{L_0} w \|_{L^2_{\lambda}}.
\end{split}
\end{equation}
Summing in dyadic numbers $N_0,N_1,N_2,L_0,L_1,L_2$, we have the bound of this case to $J_{HH \rightarrow H}$ by
$$
\lesssim \| u \|_{L^2_{\lambda}} \| v \|_{L^2_{\lambda}} \| w \|_{L^2_{\lambda}},
$$
as desired.

On the other hand, if $|\mathcal{H}(\xi_1, \xi - \xi_1, q_1, q - q_1)| \ll \xi^2$, it follows that $|(q-q_1) \xi_1 (q + q_1)| \lesssim O(\xi^2)$ and then $|q+q_1|\lesssim 1$.
In this case, we use the form
$$
q_1 = [q_1] +\left(q_1-[q_1]\right)
$$
and then
$$
q-q_1 =[q]+\left(q-[q]\right)-q_1=[q]-[q_1]+\left(q-[q]-q_1+[q_1]\right),
$$
where $[a]$ denotes integer part of $a\in\mathbb{R}$.
Note that $q_1-[q_1],~q-q_1-[q-q_1]\in \mathbb{Z}/\lambda\cap [0,1)$.
The restriction $|q+q_1|\lesssim 1$ implies $|2[q_1]+[q]|\lesssim 1$.
By performing the same calculation as \cite[Proof of Proposition 3.1]{Osawa} and \cite[Proof of Theorem 2.1]{Kenig}, we use the Cauchy-Schwarz inequality to have that the contribution of this case to $J_{HH \rightarrow H}$ is bounded by
$$
\lesssim \sup_{(\zeta,\tau)\in \mathbb{R}\times\mathbb{Z}/\lambda\times\mathbb{R}}\mathcal{I}(\zeta,\tau)\| u \|_{L^2_{\lambda}} \| v \|_{L^2_{\lambda}} \| w \|_{L^2_{\lambda}},
$$
where
$$
\mathcal{I}(\zeta,\tau)=\frac{N_0^{1-s}}{\lambda^{1/2}L_0^{1/2-}N^{1-s}}\left(\sum_{|2[q_1]+[q]|\lesssim 1}\int_{|\xi-\xi_1|\lesssim 1}\frac{d\xi_1}{\langle \tau-\sigma(\zeta)-\mathcal{H}(\xi_1,\xi-\xi_1,q_1,q-q_1)\rangle^{1+}}\right)^{1/2},
$$
where we assume $|\mathcal{H}(\xi_1, \xi - \xi_1, q_1, q - q_1)| \ll \xi^2$ in the integration of region.
Then it suffices to show
\begin{equation}\label{eq:kenig}
\sup_{(\zeta,\tau)\in \mathbb{R}\times\mathbb{Z}/\lambda\times\mathbb{R}}\mathcal{I}(\zeta,\tau)^2<\infty.
\end{equation}
The resonance function is reformulated as
\begin{equation*}
\mathcal{H} = 3\xi \xi_1 (\xi - \xi_1) + (\xi - \xi_1)(q^2 - (q-q_1)^2) + (q-q_1) \xi_1 (q+q_1),
\end{equation*}
which is equivalent to a quadratic equation in $\xi_1$,
$$
3\xi \xi_1^2-((2q-q_1)q+3\xi^2)\xi_1-(q^2-(q-q_1)^2)\xi+\mathcal{H} = 0.
$$
The roots of the quadratic equation are the values of $\xi_1$ as $\xi_1^{\pm}$, where
$$
\xi_1^{\pm}=\frac{(2q_1-q)q_1-3\xi^2\pm\sqrt{\left((2q_1-q)q-3\xi^2\right)^2+12\xi^2\left(q^2-(q-q_1)^2\right)-12\xi\mathcal{H}}}{6\xi}.
$$
Using the change of variables, we have
$$
d\xi_1 = \pm \frac{d \mathcal{H}}{\sqrt{\left((2q_1-q)q-3\xi^2\right)^2+12\xi^2\left(q^2-(q-q_1)^2\right)-12\xi\mathcal{H}}}.
$$
We evaluate integrals using the change of variables as 
\begin{equation*}
\begin{split}
&\int_{|\xi-\xi_1|\lesssim 1} \frac{d\xi_1}{ \langle \tau-\sigma(\zeta) - \mathcal{H}(\xi_1, \xi - \xi_1, q_1, q - q_1) \rangle^{1+}} \\
\lesssim & \int_{\mathbb{R}} \frac{d \mathcal{H}}{\langle \tau-\sigma(\zeta) - \mathcal{H}\rangle^{1+}\left|\left((2q_1-q)q-3\xi^2\right)^2+12\xi^2\left(q^2-(q-q_1)^2\right)-12\xi\mathcal{H}\right|^{1/2}}\\
\lesssim & \frac{1}{|\xi|^{1/2}+\left|\xi(\tau-\sigma(\zeta))+ \left((2q_1-q)q-3\xi^2\right)^2/12+\xi^2\left(q^2-(q-q_1)^2\right)\right|^{1/2}}.
\end{split}
\end{equation*}
Then we have the bound
$$
\int_{|\xi-\xi_1|\lesssim 1} \frac{d\xi_1}{ \langle \tau-\sigma(\zeta) - \mathcal{H}(\xi_1, \xi - \xi_1, q_1, q - q_1) \rangle^{1+}} \lesssim  \frac{1}{N_0^{1/2}}. 
$$
Therefore
\begin{equation}\label{eq:v-2-3}
\mathcal{I}(\zeta,\tau)\lesssim   \frac{N_0^{1-s}}{\lambda^{1/2}L_0^{1/2-}N^{1-s}}\left(\sum_{|2[q_1]+[q]|\lesssim 1}\frac{1}{N_0^{1/2}}\right)^{1/2}\lesssim \frac{1}{L_0^{1/2-}N^{1-s}N_0^{s-3/4}}\lesssim 1,
\end{equation}
which is acceptable for \eqref{eq:kenig}.

\underline{Subcase (iii):}
The proof follows from the same as one in \cite{Molinet} (also \cite[Proposition 3.1]{Osawa}).
By symmetry, we assume $||\zeta|^2-|\zeta_1|^2|\gtrsim N_0^{6/5}L_0^{0+}$.
Then
\begin{equation*}
\left| \frac{\partial \mathcal{H}}{\partial \xi_1} (\xi_1, \xi - \xi_1, q_1, q - q_1) \right|= \left||\zeta|^2-|\zeta_1|^2 \right |\gtrsim N_0^{6/5}L_0^{0+},
\end{equation*}
which modifies the computation appeared in \eqref{eq:A_1} for the proof of \eqref{eq:b2}.
Hence the contribution of this case to $J^{L_0, L_1, L_2}_{N_0, N_1, N_2}$ is bounded by
\begin{equation}\label{eq:v-3}
\begin{split}
\lesssim & \frac{N_0^{1-s}}{N^{1-s}L_0^{1/2-}L_1^{1/2+} L_2^{1/2+}} \frac {N_0^{1/2}(L_1\vee L_2)^{1/2}}{N_0^{3/5} L_0^{0+}} \| P_{N_1} Q_{L_1} u \|_{L^2_{\lambda}} \| P_{N_1} Q_{L_2} v \|_{L^2_{\lambda}} \| P_{N_0} Q_{L_0} w \|_{L^2_{\lambda}}\\
\lesssim & \frac{1}{N^{1-s}N_0^{s-9/10}L_0^{0+}L_1^{0+} L_2^{0+}}  \| P_{N_1} Q_{L_1} u \|_{L^2_{\lambda}} \| P_{N_1} Q_{L_2} v \|_{L^2_{\lambda}} \| P_{N_0} Q_{L_0} w \|_{L^2_{\lambda}},
\end{split}
\end{equation}
which is acceptable after taking the sum in dyadic numbers $N_0,N_1,N_2,L_0,L_1,L_2$.

\underline{Subcase (iv):}
We can rewrite resonance function $\mathcal{H}$ as
\begin{equation*}
\begin{split}
\mathcal{H} (\xi_1, \xi - \xi_1, q_1, q - q_1)  =  & 3 \xi \xi_1 (\xi-\xi_1) +\xi_1 q^2 - \xi q^2_1 -2 \xi_1 q q_1 + 2 \xi q q_1\\
= &-3 \xi \xi_1 (\xi-\xi_1) + P(\xi, \xi_1, q, q_1),
\end{split}
\end{equation*}
where
$$
P(\xi, \xi_1, q, q_1)=6\xi \xi_1 \xi_2 +\xi_1 q^2 - \xi q^2_1 -2 \xi_1 q q_1 + 2 \xi q q_1.
$$
By the same argument as in \cite{Molinet} (also \cite[Proposition 3.1]{Osawa}), we have in this case
\begin{equation*}
\begin{split}
\left|\mathcal{H} (\xi_1, \xi - \xi_1, q_1, q - q_1) \right| & \gtrsim (|\xi|\vee |\xi_1|\vee |\xi-\xi_1|)^2(|\xi|\wedge |\xi_1|\wedge |\xi-\xi_1|)\\
& \gtrsim (|\xi|\vee |\xi_1|\vee |\xi-\xi_1|)(|\xi|\wedge |\xi_1|\wedge |\xi-\xi_1|)|\xi|\\
& \gg N_0^{6/5}L_0^{0+}|\xi|\gtrsim |\xi|^{11/5}L_0^{0+}.
\end{split}
\end{equation*}
By symmetry we assume $L_0=L_0\vee L_1\vee L_2$, which implies $|\mathcal{H} (\xi_1, \xi - \xi_1, q_1, q - q_1)|\lesssim L_0$.
Combining these estimates in this case, we obtain $|\xi|\lesssim L_0^{5/11-}$.
We use \eqref{eq:b3} to have that the contribution of this case to $J^{L_0, L_1, L_2}_{N_0, N_1, N_2}$ is bounded by
\begin{equation}\label{eq:v-4}
\begin{split}
\lesssim & \sum_{K\lesssim L_0^{1/2--}}\frac{K}{N^{1-s}N_0^sL_0^{1/2-}L_1^{1/2+} L_2^{1/2+}} \frac{N_0^{1/2}}{K^{1/4}}(L_1\wedge L_2)^{1/2}(L_1\vee L_2)^{1/2}\\
&\quad  \| P_{N_1} Q_{L_1} u \|_{L^2_{\lambda}} \| P_{N_1} Q_{L_2} v \|_{L^2_{\lambda}} \| P_{N_0} Q_{L_0} w \|_{L^2_{\lambda}}\\
\lesssim & \frac{1}{N^{1-s}N_0^{s-9/10}L_0^{0+}L_1^{0+} L_2^{0+}}  \| P_{N_1} Q_{L_1} u \|_{L^2_{\lambda}} \| P_{N_1} Q_{L_2} v \|_{L^2_{\lambda}} \| R_KP_{N_0} Q_{L_0} w \|_{L^2_{\lambda}},
\end{split}
\end{equation}
which is acceptable after taking the sum in dyadic numbers $N_0,N_1,N_2,L_0,L_1,L_2$.

\underline{Subcase (v):}
Finally, we consider this case.
We repeat the argument in \cite{Molinet} (also \cite[Proposition 3.1]{Osawa}).
Introducing dyadic numbers $K_j$, we have that the contribution this case to $J^{L_0, L_1, L_2}_{N_0, N_1, N_2}$ is bounded by
\begin{equation}\label{eq:v1}
\lesssim \frac{1}{N^{1-s}N_0^{s}L_0^{1/2-}L_1^{1/2+}L_2^{1/2+}} \sum_{K_0, K_1,K_2} K_0\|R_{K_0}((R_{K_1}P_{N_1}Q_{L_1}u)(R_{K_2}P_{N_2}Q_{L_2}v))\|_{L^2_{\lambda}}\|R_{K_0}P_{N_0}Q_{L_0}w\|_{L^2_{\lambda}}. 
\end{equation}
In the proof of \eqref{eq:b3}, we modify that in \eqref{eq:B_1}
$$
|B^K_{\xi, q, \tau}| \lesssim \lambda \frac{(K_1 \wedge K_2) (L_1 \vee L_2)}{K^{\frac{1}{2}}},
$$
which shows
\begin{equation*}
\begin{split}
&\|R_{K_0}((R_{K_1}P_{N_1}Q_{L_1}u)(R_{K_2}P_{N_2}Q_{L_2}v))\|_{L^2_{\lambda}}\\
\lesssim & \frac{ (K_1 \wedge K_2)^{1/2} (L_1 \wedge L_2)^{1/2} (L_1 \vee L_2)^{1/4}}{K_0^{1/4}}  \| P_{N_1} Q_{L_1}u \|_{L^2_{\lambda}} \| P_{N_2} Q_{L_2}v \|_{L^2_{\lambda}}.
\end{split}
\end{equation*}
Then the left hand-side of \eqref{eq:v1} is controlled by
\begin{equation*}
\lesssim   \sum_{K_0, K_1,K_2} \frac{K_0^{3/4} (K_1 \wedge K_2)^{1/2} (L_1 \wedge L_2)^{1/2} (L_1 \vee L_2)^{1/4}}{N^{1-s}N_0^{s}L_0^{1/2-}L_1^{1/2+}L_2^{1/2+}}   \| P_{N_1} Q_{L_1} u \|_{L^2_{\lambda}} \| P_{N_2} Q_{L_2} v \|_{L^2_{\lambda}}\|P_{N_0}Q_{L_0}w\|_{L^2_{\lambda}},
\end{equation*}
where the sum is based on the region $(K_0\vee K_1\vee K_2)(K_0\wedge K_1\wedge K_2)\lesssim N_0^{6/5}(L_0\vee L_1\vee L_1)^{0+}$.
Since $K_0^{3/4} (K_1 \wedge K_2)^{1/2} \lesssim N_0^{9/10}(L_0\vee L_1\vee L_2)^{0+}$, we have that the contribution of this case to $J^{L_0, L_1, L_2}_{N_0, N_1, N_2}$ is bounded by
\begin{equation}\label{eq:v-5}
\lesssim  \frac{1}{N^{1-s}N_0^{s-9/10}L_0^{0+}L_1^{0+}L_2^{0+}}  \| P_{N_1} Q_{L_1} R_{K_1} u \|_{L^2_{\lambda}} \| P_{N_2} Q_{L_2} R_{K_2} v \|_{L^2_{\lambda}}\|R_{K_0}P_{N_0}Q_{L_0}w\|_{L^2_{\lambda}},
\end{equation}
which is acceptable.

We finish the proof of $J_{HH \rightarrow H}$, and hence Lemma \ref{prop:I-b}.
\end{proof}

We are now in position to prove Proposition \ref{prop:I-local}.
\begin{proof}[Proof of Proposition \ref{prop:I-local}.]
Define
\begin{equation*}
\Psi(u^{\lambda})(t)=\eta(t)e^{-t \partial_x \Delta} u_0^{\lambda}+\frac12\eta(t)\int_0^t e^{-(t-t') \partial_x \Delta}\partial_x\left(u^{\lambda}(t')^2\right)\,dt'.
\end{equation*}
We show that the map $\Psi$ defines a contraction in
$$
Y_{T,\lambda}=\left\{u^{\lambda}\in I^{-1}X^{1,1/2+}_{T,\lambda}\mid \|Iu^{\lambda}\|_{X^{1,1/2+}_{T,\lambda}}\le 2C\|Iu_0^{\lambda}\|_{H^1_{\lambda}}\right\},
$$
for (small) $T>0$, where the norm on $Y_{T,\lambda}$ is induced by $\|Iu^{\lambda}\|_{X^{1,1/2+}_{T,\lambda}}$,
In fact, from Lemmas \ref{lem:bourgain1}, \ref{lem:bourgain2} and \ref{lem:bourgain3} and \ref{prop:I-b}, it follows that
\begin{equation*}
\begin{split}
\|I\Psi(u^{\lambda})\|_{X^{1,1/2+}_{T,\lambda}} & \le C\|Iu_0^{\lambda}\|_{H^1_{\lambda}}+C\left\|\partial_xI(u^{\lambda})^2\right\|_{X^{1,-1/2+\varepsilon}_{T,\lambda}}\\
& \le C\|Iu_0^{\lambda}\|_{H^1_{\lambda}}+CT^{\varepsilon}\left\|\partial_xI(u^{\lambda})^2\right\|_{X^{s,-1/2+2\varepsilon}_{T,\lambda}}\\
& \le   C\|Iu_0^{\lambda}\|_{H^1_{\lambda}}+CT^{\varepsilon}\left\|Iu^{\lambda}\right\|_{X^{s,-1/2+}_{T,\lambda}}^2\\
& \le 2C\|Iu_0^{\lambda}\|_{H^1_{\lambda}},
\end{split}
\end{equation*}
for $u^{\lambda}\in Y_{T,\lambda}$ and choosing small $T>0$ which will depend on $\|Iu_0^{\lambda}\|_{H^s_{\lambda}}$.
Similarly,
$$
\left\|I\Psi(u^{\lambda})-I\Psi(v^{\lambda})\right\|_{X^{1,1/2+}_{T,\lambda}}\le \frac{1}{2}\left\|Iu^{\lambda}-Iv^{\lambda}\right\|_{X^{1,1/2+}_{T,\lambda}},
$$
for $u^{\lambda},v^{\lambda}\in Y_{T,\lambda}$ and small $T>0$.
Then the contraction mapping theorem tells us that there is a unique solution $Iu^{\lambda}=I\Psi(u^{\lambda})\in Y_{T,\lambda}$ to the Cauchy problem \eqref{eq:lambdaZK}.

The persistence properties $u^{\lambda}\in C([-\delta,\delta]:H^s_{\lambda})$ and the uniqueness in whole space $Iu^{\lambda}\in X^{1,1/2+}_{T,\lambda}$ are follow a similar argument to \cite[Proof of Theorem 1.5]{Kenig}, therefore we omit them.
\end{proof}

\section{Modified energy}\label{sec:energy}
\indent
  
The conservation quantities associated with \eqref{eq:lambdaZK} are
$$
E^{\lambda}[u^{\lambda}](t) = \frac{1}{2} \int_{\mathbb{R} \times \mathbb{T_{\lambda}}} \left({|\nabla u^{\lambda} (x,y,t)|}^2 - \frac{1}{3} {u^{\lambda}(x,y,t)}^3\right) \,dxdy,
$$
$$  
M^{\lambda}[u^{\lambda}](t) = \int_{\mathbb{R} \times \mathbb{T_{\lambda}}} {u^{\lambda}(x,y,t)}^2 \,dxdy.
$$

From $\mathcal{F}^{\lambda}u_0^{\lambda}(\xi,q)=\mathcal{F}u_0(\lambda\xi,\lambda q)$, it is easy to see
\begin{equation}\label{eq:lambdaL^2}
\left\|Iu^{\lambda}_0\right\|_{L^2_{\lambda}}\le \left\|u_0^{\lambda}\right\|_{L^2_{\lambda}}\le \frac{\left\|u_0\right\|_{L^2}}{\lambda}=o(1)
\end{equation}
for $\lambda\gg 1$.
Here the choice of the large parameter $N$ will be made latter, but $\lambda>1$ is chosen by
\begin{equation}\label{eq:choice}
\frac{N^{1-s}}{\lambda^{1+s}}=o(1),
\end{equation}
in which 
\begin{equation}\label{eq:nabla}
\left\|\nabla Iu^{\lambda}_0\right\|_{L^2_{\lambda}}\lesssim \frac{N^{1-s}}{\lambda^{1+s}}\|u_0\|_{H^s}=o(1),\quad \|Iu_0^{\lambda}\|_{H^1_{\lambda}}=o(1).
\end{equation}
Moreover, we have
\begin{equation}\label{eq:lambda-s}
\left\|u\left(\frac{t}{\lambda^3}\right)\right\|_{H^s}\lesssim \lambda^{1+s}\left\|Iu^{\lambda}(t)\right\|_{H^s_{\lambda}}\le \lambda^{1+s}\left\|Iu^{\lambda}(t)\right\|_{H^1_{\lambda}}.
\end{equation}

\begin{Remark}
By Gagliardo-Nirenberg inequality, the conservation law of $L^2_{\lambda}$-norm and \eqref{eq:lambdaL^2}, the solution $u^{\lambda}(t)$ satisfies
\begin{equation}\label{eq:gn}
\left\|\nabla Iu^{\lambda}(t)\right\|_{L^2_{\lambda}}^2 -C_1\left\|u_0^{\lambda}\right\|_{L^2_{\lambda}}^4\le C_2E^{\lambda}[Iu^{\lambda}](t),
\end{equation}
where constants $C_1$ and $C_2$ are independent of $\lambda$. 
\end{Remark}

Let us introduce the modified energy for proving global well-posedness.
Using the Fundamental Theorem of Calculus, we get
$$
E^{\lambda}[Iu^{\lambda}](\delta) - E^{\lambda}[Iu^{\lambda}](0) = \int^{\delta}_0 \frac{dE^{\lambda}[Iu^{\lambda}](t)}{dt}\,dt
$$
for $\delta>0$. We continue calculation of the integrand in the right hand.
Then
\begin{equation}\label{eq:de}
\begin{split}
\frac{dE^{\lambda}[Iu^{\lambda}]}{dt} (t) &= \int_{\mathbb{R} \times \mathbb{T}_{\lambda}} \left(\frac{d}{dt}\nabla Iu^{\lambda} \cdot \nabla Iu^{\lambda}-\frac12 (Iu^{\lambda})^2 \frac{d}{dt}Iu^{\lambda}\right)\, dxdy\\
&= \int_{\mathbb{R} \times \mathbb{T}_{\lambda}}\left( -I \partial_t u^{\lambda}\Delta Iu^{\lambda} - \frac12I \partial_t u^{\lambda}(Iu^{\lambda})^2\right)\,dxdy\\
&= \int_{\mathbb{R} \times \mathbb{T}_{\lambda}} \partial_x\left(I \Delta u^{\lambda} + \frac12I (u^{\lambda})^2\right)\left(\frac12(Iu^{\lambda})^2 + \Delta I u^{\lambda}\right)\,dxdy\\
&= -\frac12\int_{\mathbb{R} \times \mathbb{T}_{\lambda}} \left( I \Delta u^{\lambda} \partial_x\left((Iu^{\lambda})^2 - I(u^{\lambda})^2\right) +\frac12 I(u^{\lambda})^2\partial_x\left((Iu^{\lambda})^2 - I(u^{\lambda})^2\right)\right)\,dxdy
\end{split}
\end{equation}
where we use the equation in \eqref{eq:lambdaZK}.
Taking the integral on $[0,\delta]$, we have
\begin{equation}\label{eq:ie}
\begin{split}
E^{\lambda}[Iu^{\lambda}](\delta) - E^{\lambda}[Iu^{\lambda}](0)  = & -\frac12\int_0^{\delta}\!\int_{\mathbb{R} \times \mathbb{T}_{\lambda}}I \Delta u^{\lambda}\partial_x\left((Iu^{\lambda})^2 - I(u^{\lambda})^2\right)\,dxdydt\\
&-\frac14\int_0^{\delta}\!\int_{\mathbb{R} \times \mathbb{T}_{\lambda}}I(u^{\lambda})^2\partial_x\left((Iu^{\lambda})^2 - I(u^{\lambda})^2\right)\,dxdydt
\end{split}
\end{equation}
We will estimate two terms on the right hand of \eqref{eq:ie} with $\|Iu^{\lambda}\|_{X^{1,1/2+}_{\lambda,\delta}}$ by multilinear estimates associated with the transition of energy in \eqref{eq:ie}.

\begin{lemma}\label{lem:ibilinear}
For $s>9/10$,
\begin{equation}\label{eq:ibilinear}
\left\| \partial_x \left(IuIv - I(uv)\right) \right\|_{X^{1,-1/2+}_{\lambda}} \lesssim N^{-1/10+}  \| Iu \|_{X^{1,1/2+}_{\lambda}}\| Iv \|_{X^{1,1/2+}_{\lambda}}.
\end{equation}
\end{lemma}

\begin{Remark}
Comparing to the estimate in Lemma \ref{prop:I-b}, we have the small factor at the front of the right-hand, which corresponding to properties such as dispersion or the smoothing effects.
\end{Remark}

\begin{proof}
We recast the proof of Lemma \ref{prop:I-b} in the form of formula \eqref{eq:ibilinear}.
Here we used the same symbol over as in the proof of Lemma \ref{prop:I-b}. 
By Plancherel identity, it suffices to show that
\begin{equation}\label{eq:iI-b1}
\begin{split}
J= & \frac{1}{\lambda^2}\sum_{q,q_1 \in \mathbb{Z}/\lambda} \int_{\mathbb{R}^4} \Gamma^{\xi_1, q_1, \tau_1}_{\xi, q, \tau}\frac{m(\zeta)}{m(\zeta_1)m(\zeta-\zeta_1)}  \Xi(\zeta,\zeta_1,\zeta-\zeta_1)\\
& \qquad \widehat{u}^{\lambda}(\zeta_1, \tau_1) \widehat{v}^{\lambda}(\zeta-\zeta_1, \tau-\tau_1) \widehat{w}^{\lambda}(\zeta, \tau)\, d\xi d\xi_1 d \tau d \tau_1\\
\lesssim & \| u \|_{L^2_{\lambda}} \| v \|_{L^2_{\lambda}} \| w \|_{L^2_{\lambda}},
\end{split}
\end{equation}
for $u,v,w\in L^2_{\lambda}$ whose Fourier transform are nonnegative.
In \eqref{eq:iI-b1},
$$
\Gamma^{\xi_1, q_1, \tau_1}_{\xi, q, \tau} =\frac{ |\xi| \langle \zeta \rangle }{ \langle \zeta_1 \rangle \langle\zeta- \zeta_1 \rangle \langle \tau-\sigma(\zeta) \rangle^{1/2-} \langle \tau_1-\sigma(\zeta_1)\rangle^{1/2+} \langle \tau-\tau_1-\sigma(\zeta-\zeta_1) \rangle^{1/2+}},
$$
$$
\Xi(\zeta,\zeta_1,\zeta-\zeta_1)=\frac{N^{1/10-}}{m(\zeta)}\left| m(\zeta_1)m(\zeta-\zeta_1)-m(\zeta)\right|.
$$
Note the trivial bound
\begin{equation}\label{eq:Xi}
\Xi(\zeta,\zeta_1,\zeta-\zeta_1)\lesssim \frac{N^{1/10-}}{m(\zeta)}.
\end{equation}
Using dyadic decomposition in a similar way to the proof of Lemma \ref{prop:I-b}, we rewrite $J$ as the following 
\begin{equation*}
J = \sum_{\scriptstyle N_0,N_1,N_2 \atop\scriptstyle L_0, L_1, L_2} J^{L_0, L_1, L_2}_{N_0, N_1, N_2},
\end{equation*}
where
\begin{equation}\label{eq:iJQ}
\begin{split}
J^{L_0, L_1, L_2}_{N_0, N_1, N_2} = & \frac{1}{\lambda^2}\sum_{q,q_1\in\mathbb{Z}/\lambda} \int_{\mathbb{R}^4} \Gamma^{\xi_1, q_1, \tau_1}_{\xi, q, \tau}\frac{m(\zeta)}{m(\zeta_1)m(\zeta-\zeta_1)}\Xi(\zeta,\zeta_1,\zeta-\zeta_1) \\ 
& \widehat{P_{N_1} Q_{L_1} u}^{\lambda}(\zeta_1, \tau_1)  \widehat{P_{N_2} Q_{L_2} v}^{\lambda}(\zeta-\zeta_1, \tau-\tau_1)\widehat{P_{N_0} Q_{L_0} w}^{\lambda}(\xi, q, \tau)\,d\xi d\xi_1 d\tau d\tau_1.
\end{split}
\end{equation}
We repeat the same procedure as one of Lemma \ref{prop:I-b}.
Decompose again $J$ to five parts in frequencies:
$$
J_{LL \rightarrow L} = \sum_{\scriptstyle N_1\vee N_2 \vee N_0 \ll N \atop{\scriptstyle L_0,L_1,L_2}} J^{L_0, L_1, L_2}_{N_0, N_1, N_2},
$$
$$
J_{LH \rightarrow H} = \sum_{\scriptstyle N_1\ll N_2\sim N_0 \atop{\scriptstyle N_0 \gtrsim N \atop{\scriptstyle L_0,L_1,L_2}}} J^{L_0, L_1, L_2}_{N_0, N_1, N_2},
$$
$$
J_{HL \rightarrow H} = \sum_{\scriptstyle N_2\ll N_1\sim N_0  \atop{\scriptstyle N_0 \gtrsim N \atop{\scriptstyle L_0,L_1,L_2}}} J^{L_0, L_1, L_2}_{N_0, N_1, N_2},
$$
$$
J_{HH \rightarrow L} = \sum_{\scriptstyle N_0 \ll N_1\sim N_2  \atop{\scriptstyle N_1 \gtrsim N \atop{\scriptstyle L_0,L_1,L_2}}}J^{L_0, L_1, L_2}_{N_0, N_1, N_2},
$$
$$
J_{HH \rightarrow H} = J - (J_{LL \rightarrow L} + J_{HL \rightarrow H} + J_{LH \rightarrow H} + J_{HH \rightarrow L}).
$$
We estimate each of them by case analysis.

\underline{Estimate of $J_{LL \rightarrow L}$:}
In this region, we have $\Xi(\zeta,\zeta_1,\zeta-\zeta_1)=0$, since by $m(\zeta)=m(\zeta_1)=m(\zeta-\zeta_1)=1$.
Hence $J_{LL \rightarrow L}=0$.

\underline{Estimate of $J_{LH \rightarrow H}$:}
We split the case into two cases, when $N_1\gtrsim N$ and when $N_1\ll N$.
When $N_1\gtrsim N$, we use the bound \eqref{eq:Xi}.
On the other hand, when $N_1\ll N$, the mean value theorem gives
\begin{equation}\label{eq:mean}
|m(\zeta_1)m(\zeta-\zeta_1)-m(\zeta)|=|m(\zeta-\zeta_1)-m(\zeta)|\lesssim m'(N_0)N_1\sim \frac{m(N_0)N_1}{N_0}.
\end{equation}
By making use of the estimates in \eqref{eq:LH1} and \eqref{eq:LH2}, we easily have the desired bounds for the contribution of this case to $J_{LH \rightarrow H}$.
Indeed, when $N_1\ll N$, the computation in \eqref{eq:LH1} with \eqref{eq:Xi} leads the bound
\begin{equation*}
\begin{split}
\lesssim & \frac{N_1}{N_0}\frac{N^{1/10-}}{N_1^{1/2} L_0^{1/2-} L_1^{0+} L_2^{0+}} \|P_{N_1} Q_{L_1} u\|_{L^2_{\lambda}} \|P_{N_2} Q_{L_2} v\|_{L^2_{\lambda}} \| P_{N_0} Q_{L_0} w \|_{L^2_{\lambda}}\\
\lesssim & \frac{1}{N_1^{0+} L_0^{1/2-} L_1^{0+} L_2^{0+}} \|P_{N_1} Q_{L_1} u\|_{L^2_{\lambda}} \|P_{N_2} Q_{L_2} v\|_{L^2_{\lambda}} \| P_{N_0} Q_{L_0} w \|_{L^2_{\lambda}}.
\end{split}
\end{equation*}
While $N_1\gtrsim N$, we have $|m(\zeta_1)m(\zeta-\zeta_1)-m(\zeta)|\lesssim m(N_0)$ and then $\Xi(\zeta,\zeta_1,\zeta-\zeta_1)\lesssim N^{1/10-}$.
By the computation in \eqref{eq:LH2} with \eqref{eq:mean} we estimate this contribution by
\begin{equation*}
\begin{split}
\lesssim &  N^{1/10-}\frac{1}{N_1^{s-1/2} L_0^{1/2-} L_1^{0+} L_2^{0+}} \|P_{N_1} Q_{L_1} u\|_{L^2_{\lambda}} \|P_{N_2} Q_{L_2} v\|_{L^2_{\lambda}} \| P_{N_0} Q_{L_0} w \|_{L^2_{\lambda}}\\
\lesssim & \frac{1}{N_1^{0+} L_0^{0+} L_1^{0+} L_2^{0+}} \|P_{N_1} Q_{L_1} u\|_{L^2_{\lambda}} \|P_{N_2} Q_{L_2} v\|_{L^2_{\lambda}} \| P_{N_0} Q_{L_0} w \|_{L^2_{\lambda}}.
\end{split}
\end{equation*}
Then the desired estimate follows.

\underline{Estimate of $J_{HL \rightarrow H}$:}
We use symmetry arguments as for $J_{LH \rightarrow H}$ to conclude the proof.

\underline{Estimate of $J_{HH \rightarrow L}$:}
In this case, we use the bound \eqref{eq:Xi}.
By employing the same estimates as in \eqref{eq:HH1} and \eqref{eq:HH2}, we have the desired bounds.
Actually, the computation in \eqref{eq:HH1} will be changed into
\begin{equation*}
\begin{split}
\lesssim &\frac{N^{1/10-}}{N^{1-s}N_1^{s-1/2-3\theta/2}L_0^{0+}L_1^{0+}L_2^{0+}}  \| P_{N_1} Q_{L_1} u \|_{L^2_{\lambda}} \| P_{N_0} Q_{L_0} w \|_{L^2_{\lambda}} \| P_{N_2} Q_{L_2} v \|_{L^2_{\lambda}}\\
\lesssim &\frac{1}{N_1^{2/5-3\theta/2}L_0^{0+}L_1^{0+}L_2^{0+}}  \| P_{N_1} Q_{L_1} u \|_{L^2_{\lambda}} \| P_{N_0} Q_{L_0} w \|_{L^2_{\lambda}} \| P_{N_2} Q_{L_2} v \|_{L^2_{\lambda}},
\end{split}
\end{equation*}
for small $\theta>0$, while the formula in \eqref{eq:HH2} is changed by the same.

\underline{Estimate of $J_{HH \rightarrow H}$:}
In this case, we can assume $N\ll N_0\sim N_1\sim N_2$ and then
$$
\Xi(\zeta,\zeta_1,\zeta-\zeta_1)\lesssim  N^{1/10-}.
$$
We divide the region into five regions as in the proof of Lemma \ref{prop:I-b}.
On each case, we recall the estimates in \eqref{eq:v-1}, \eqref{eq:v-2-1}, \eqref{eq:v-2-2}, \eqref{eq:v-2-3},  \eqref{eq:v-3}, \eqref{eq:v-4} and \eqref{eq:v-5}, by multiplying $\Xi(\zeta,\zeta_1,\zeta-\zeta_1)$.
We see that $J_{HH \rightarrow H}\lesssim  \| u \|_{L^2_{\lambda}} \| v \|_{L^2_{\lambda}} \| w \|_{L^2_{\lambda}}$ holds provided $s>9/10$. 

In fact, in subcase (i), the estimate corresponding to \eqref{eq:v-1} will be
\begin{equation*}
\begin{split}
\lesssim   & N^{1/10-}\frac{2^{-3k/4}}{N^{1-s}N_0^{s-1/2}L_0^{0+}L_1^{0+}L_2^{0+}} \| P_{N_1} Q_{L_1} u \|_{L^2_{\lambda}} \| P_{N_2} Q_{L_2} v \|_{L^2_{\lambda}} \| P_{N_0} Q_{L_0} w \|_{L^2_{\lambda}}\\
\lesssim & \frac{2^{-3k/4}}{N_0^{2/5-}L_0^{0+}L_1^{0+}L_2^{0+}} \| P_{N_1} Q_{L_1} u \|_{L^2_{\lambda}} \| P_{N_2} Q_{L_2} v \|_{L^2_{\lambda}} \| P_{N_0} Q_{L_0} w \|_{L^2_{\lambda}},
\end{split}
\end{equation*}
which is acceptable after summing over $N_0,N_1,N_2,L_0,L_1,L_2$.

In subcase (ii), we revisit the estimates in \eqref{eq:v-2-1}, \eqref{eq:v-2-2} and \eqref{eq:v-2-3}.
We require the following estimate corresponding to \eqref{eq:v-2-1}
\begin{equation*}
\begin{split}
\lesssim &   N^{1/10-}\frac{1}{N^{1-s}N_0^{s-1/2-}L_0^{0+} L_1^{0+} L_2^{0+}}\| P_{N_1} Q_{L_1} u \|_{L^2_{\lambda}} \| P_{N_2} Q_{L_2} v \|_{L^2_{\lambda}} \|P_{N_0} Q_{L_0} w \|_{L^2_{\lambda}} \\
\lesssim & \frac{1}{N_0^{2/5-}L_0^{0+} L_1^{0+} L_2^{0+}}\| P_{N_1} Q_{L_1} u \|_{L^2_{\lambda}} \| P_{N_2} Q_{L_2} v \|_{L^2_{\lambda}} \|P_{N_0} Q_{L_0} w \|_{L^2_{\lambda}}.
\end{split}
\end{equation*}
The corresponding estimate to \eqref{eq:v-2-2} is
\begin{equation*}
\begin{split}
\lesssim & N^{1/10-} \frac{1}{N^{1-s}N_0^{s-} L_0^{0+} L_1^{0+} L_2^{0+}} \|P_{N_1} Q_{L_1} u\|_{L^2_{\lambda}} \|P_{N_2} Q_{L_2} v\|_{L^2_{\lambda}} \| P_{N_0} Q_{L_0} w \|_{L^2_{\lambda}}\\
\sim & \frac{1}{N_0^{9/10-} L_0^{0+} L_1^{0+} L_2^{0+}} \|P_{N_1} Q_{L_1} u\|_{L^2_{\lambda}} \|P_{N_2} Q_{L_2} v\|_{L^2_{\lambda}} \| P_{N_0} Q_{L_0} w \|_{L^2_{\lambda}}.
\end{split}
\end{equation*}
The corresponding estimate to \eqref{eq:v-2-3} is 
\begin{equation*}
\mathcal{I}(\zeta,\tau)\lesssim N^{1/10-} \frac{1}{L_0^{1/2-}N^{1-s}N_0^{s-3/4}}\lesssim 1.
\end{equation*}

By a similar proof to one of Lemma \ref{prop:I-b}, we have a desired bound. 

In subcase (iii), it suffices to show the following estimate corresponding to \eqref{eq:v-3} 
\begin{equation*}
\begin{split}
\lesssim & N^{1/10-}\frac{N_0^{1-s}}{N^{1-s}L_0^{1/2-}L_1^{1/2+} L_2^{1/2+}} \frac {N_0^{1/2}(L_1\vee L_2)^{1/2}}{N_0^{3/5} L_0^{0+}} \| P_{N_1} Q_{L_1} u \|_{L^2_{\lambda}} \| P_{N_1} Q_{L_2} v \|_{L^2_{\lambda}} \| P_{N_0} Q_{L_0} w \|_{L^2_{\lambda}}\\
\lesssim & \frac{1}{N_0^{0+}L_0^{0+}L_1^{0+} L_2^{0+}}  \| P_{N_1} Q_{L_1} u \|_{L^2_{\lambda}} \| P_{N_1} Q_{L_2} v \|_{L^2_{\lambda}} \| P_{N_0} Q_{L_0} w \|_{L^2_{\lambda}},
\end{split}
\end{equation*}
which holds for $s>9/10$.

In subcase (iv), the corresponding estimate to \eqref{eq:v-4} is 
\begin{equation*}
\begin{split}
\lesssim & N^{1/10-}\frac{1}{N^{1-s}N_0^{s-9/10}L_0^{0+}L_1^{0+} L_2^{0+}}  \| P_{N_1} Q_{L_1} u \|_{L^2_{\lambda}} \| P_{N_1} Q_{L_2} v \|_{L^2_{\lambda}} \| R_KP_{N_0} Q_{L_0} w \|_{L^2_{\lambda}}\\
\lesssim & \frac{1}{N_0^{0+}L_0^{0+}L_1^{0+} L_2^{0+}}  \| P_{N_1} Q_{L_1} u \|_{L^2_{\lambda}} \| P_{N_1} Q_{L_2} v \|_{L^2_{\lambda}} \| R_KP_{N_0} Q_{L_0} w \|_{L^2_{\lambda}},
\end{split}
\end{equation*}
which is acceptable for $s>9/10$.

Finally, in subcase (v), we need the following estimate corresponding estimate to \eqref{eq:v-5}
\begin{equation*}
\begin{split}
\lesssim & N^{1/10-}\frac{1}{N^{1-s}N_0^{s-9/10}L_0^{0+}L_1^{0+}L_2^{0+}}  \| P_{N_1} Q_{L_1} R_{K_1} u \|_{L^2_{\lambda}} \| P_{N_2} Q_{L_2} R_{K_2} v \|_{L^2_{\lambda}}\|R_{K_0}P_{N_0}Q_{L_0}w\|_{L^2_{\lambda}}\\
\lesssim &  \frac{1}{N_0^{0+}L_0^{0+}L_1^{0+}L_2^{0+}}  \| P_{N_1} Q_{L_1} R_{K_1} u \|_{L^2_{\lambda}} \| P_{N_2} Q_{L_2} R_{K_2} v \|_{L^2_{\lambda}}\|R_{K_0}P_{N_0}Q_{L_0}w\|_{L^2_{\lambda}},
\end{split}
\end{equation*}
which is also acceptable for $s>9/10$.

Hence we complete the proof. 
\end{proof}

By Lemma \ref{lem:ibilinear}, we have the following lemma.

\begin{lemma}\label{lem:3multi}
For $s>9/10$,
\begin{equation*}
\left|\int_0^{\delta}\!\int_{\mathbb{R} \times \mathbb{T}_{\lambda}}I \Delta u\partial_x\left((Iu)^2 - I(u)^2\right)\,dxdydt\right|\lesssim \frac{1}{N^{1/10-}}\left\|Iu\right\|_{X^{1,1/2+}_{\delta,\lambda}}^3.
\end{equation*}
\end{lemma}

\begin{proof}

We apply the Parseval's identity and the Cauchy-Schwartz inequality to get the bound of the left-hand side by
$$
\lesssim \left\|I\Delta u\right\|_{X^{-1,1/2+}_{\delta,\lambda}}\left\|\partial_x(Iu)^2 - I(u)^2)\right\|_{X^{1,1/2-}_{\delta,\lambda}}\lesssim \left\|Iu\right\|_{X^{1,1/2+}_{\delta,\lambda}}\left\|\partial_x(Iu)^2 - I(u)^2)\right\|_{X^{1,1/2-}_{\delta,\lambda}}.
$$
Employing Lemma \ref{lem:ibilinear}, we have the desired bound.
\end{proof} 

\begin{lemma}\label{lem:4th}
For $s>31/40$,
\begin{equation*}
\left|\int_{\mathbb{R} \times \mathbb{T}_{\lambda}\times\mathbb{R}} Iu^2\partial_x\left((Iu)^2 - I(u)^2\right)\,dxdydt\right|\lesssim \frac{1}{N^{1/10+}}\|Iu\|_{X^{1,1/2+}_{\lambda}}^4.
\end{equation*}
\end{lemma}

\begin{proof}
We apply Parseval's identity to the left-hand side and use the same argument as in the proof of Lemma \ref{prop:I-b}.

It suffices to show that
\begin{equation}\label{eq:4th1}
J=  \frac{1}{\lambda^3}\sum_{*} \int_{*} |\xi_3+\xi_4|\Gamma_{\zeta_1,\zeta_2,\zeta_3,\zeta_4}^{\tau_1,\tau_2,\tau_3,\tau_4}\Xi(\zeta_1,\zeta_2,\zeta_3,\zeta_4)\prod_{j=1}^4\widehat{u}^{\lambda}(\zeta_j, \tau_j)\lesssim  \| u \|_{L^2_{\lambda}}^4,
\end{equation}
for $u\in L^2_{\lambda}$ whose Fourier transform is nonnegative.
In \eqref{eq:4th1}, $\sum_{*} \int_{*}$ stands for the sum-integral of $(\xi_1,\xi_2,\xi_3,q_1,q_2,q_3,\tau_1,\tau_2,\tau_3)\in \mathbb{R}^3\times (\mathbb{Z}/\lambda)^3\times\mathbb{R}^3$ over $\zeta_1+\zeta_2+\zeta_3+\zeta_4=0$, and 
$$
\Gamma_{\zeta_1,\zeta_2,\zeta_3,\zeta_4}^{\tau_1,\tau_2,\tau_3,\tau_4}=\prod_{j=1}^4\frac{1}{m(\zeta_j)\langle \zeta_j\rangle\langle \tau_j-\sigma(\zeta_j) \rangle^{1/2+}},
$$
$$
\Xi(\zeta_1,\zeta_2,\zeta_3,\zeta_4)=m(\zeta_1+\zeta_2)\left| m(\zeta_3)m(\zeta_4)-m(\zeta_3+\zeta_4)\right|N^{1/10+}.
$$
Note the trivial bound
\begin{equation}\label{eq:4th1t}
\Xi(\zeta_1,\zeta_2,\zeta_3,\zeta_4)\lesssim N^{1/10+}.
\end{equation}
By dyadic decomposition in a similar way to the proof of Lemma \ref{prop:I-b}, we use $\langle \zeta_j\rangle \sim N_j$ and $\langle\tau_j-\sigma(\zeta_j)\rangle \sim L_j$, as before:
\begin{equation*}
J = \sum_{\scriptstyle N_1,N_2,N_3,N_4 \atop\scriptstyle L_1, L_2, L_3,L_4} J^{L_1, L_2, L_3,L_4}_{N_1, N_2, N_3,N_4},
\end{equation*}
where
\begin{equation*}
J^{L_1, L_2, L_3,L_4}_{N_1, N_2, N_3,N_4}=  \frac{1}{\lambda^3}\sum_{*} \int_{*} |\xi_3+\xi_4|\Gamma_{\zeta_1,\zeta_2,\zeta_3,\zeta_4}^{\tau_1,\tau_2,\tau_3,\tau_4}\Xi(\zeta_1,\zeta_2,\zeta_3,\zeta_4)\prod_{j=1}^4\widehat{P_{N_j} Q_{L_j} u}^{\lambda}(\zeta_j, \tau_j).\end{equation*}

We will show \eqref{eq:4th1} by the case-by-case analysis.

In case $N_3\vee N_4\ll N$, we have $\Xi(\zeta_1,\zeta_2,\zeta_3,\zeta_4)=0$, which is canceled out when evaluating \eqref{eq:4th1}.
Then we may assume $N_3\vee N_4\gtrsim N$ in \eqref{eq:4th1}.

Note that $|\xi_3+\xi_4|=|\xi_1+\xi_2|\lesssim (N_1\vee N_2)\wedge (N_3\vee N_4)$.
We take $(P_{N_1} Q_{L_1} u)(P_{N_2} Q_{L_2} u)$ and $(P_{N_3} Q_{L_3} u)(P_{N_4} Q_{L_4} u)$ in $L^2_{\lambda}$ by using \eqref{eq:b1}, respectively.
By \eqref{eq:4th1t},  the contribution of this case to $J^{L_1, L_2, L_3,L_4}_{N_1, N_2, N_3,N_4}$ is bounded by
\begin{equation*}
\begin{split}
\lesssim & \frac{(N_1\wedge N_2)(N_1\vee N_2)^{11/20}}{N_1m(N_1)N_2m(N_2)L_1^{0+}L_2^{0+}}\frac{(N_3\wedge N_4)(N_3\vee N_4)^{9/20}}{N_3m(N_3)N_4m(N_4)L_3^{0+}L_4^{0+}}N^{1/10+}\prod_{j=1}^4\|P_{N_j} Q_{L_j} u\|_{L^2_{\lambda}}\\
\lesssim & \prod_{j=1}^4\frac{1}{N_j^{9/40-}m(N_j)L_j^{0+}}\|P_{N_j} Q_{L_j} u\|_{L^2_{\lambda}}\\
\lesssim & \prod_{j=1}^4\frac{1}{N_j^{0+}L_j^{0+}}\|P_{N_j} Q_{L_j} u\|_{L^2_{\lambda}},
\end{split}
\end{equation*}
since the function $f(x)=x^{9/40-}m(x)$ is non-decreasing on $[1,\infty)$ provided $s>31/40$.
This is acceptable after taking the dyadic sum on $N_j$ and $L_j$.
\end{proof}

The following Lemma is not difficult to prove using Lemma \ref{lem:4th} as in \cite{Colliander,Colliander1}.

\begin{lemma}\label{lem:4multi}
For $s>31/40(<29/31)$,
\begin{equation*}
\left|\int_0^{\delta}\!\int_{\mathbb{R} \times \mathbb{T}_{\lambda}} I(u)^2\partial_x\left((Iu)^2 - I(u)^2\right)\,dxdydt\right|\lesssim \frac{1}{N^{1/10+}}\left\|Iu\right\|_{X^{1,1/2+}_{\delta,\lambda}}^4.
\end{equation*}
\end{lemma}

\section{Proof of main theorem}\label{sec:GWP}
\indent

In this section, we will give the proof of Theorem \ref{thm:main}, by the local well-posedness results of Proposition \ref{prop:I-local} in conjunction with Lemmas \ref{lem:3multi} and \ref{lem:4multi}.
The proof follows the strategy described in \cite{Colliander}.
We prove the global well-posedness by the local well-posedness theory for solution $Iu$ obtained in Proposition \ref{prop:I-local} and by the upper bound of increment of modified energy from \eqref{eq:ie}.

\begin{proof}[Proof of Theorem \ref{thm:main}]
Consider the Cauchy problem \eqref{ZK} on an arbitrary time interval $[0,T]$.
We rescale solutions by scaling \eqref{eq:scaling} and consider the rescaled problem in \eqref{eq:lambdaZK} with initial data
$$
u_0^{\lambda}(x,y)=\frac{1}{\lambda^2}u_0\left(\frac{x}{\lambda},\frac{y}{\lambda}\right).
$$
The goal is to construct rescaled solution $u^{\lambda}(t)$ to the Cauchy problem \eqref{eq:lambdaZK} on the time interval $[0,\lambda^3 T]$.

By the local well-posedness result of Proposition \ref{prop:I-local} along with \eqref{eq:nabla}, our solution $u^{\lambda}(t)$ satisfies $\|Iu^{\lambda}\|_{X^{1,1/2+}_{\delta,\lambda}}\ll 1$.
Applying Lemmas \ref{lem:3multi} and \ref{lem:4multi} to the computation of \eqref{eq:ie}, we have
$$
E^{\lambda}[Iu^{\lambda}](\delta)\le E^{\lambda}[Iu^{\lambda}](0)+\frac{o(1)}{N^{1/10-}}=o(1).
$$
By using \eqref{eq:lambdaL^2} and \eqref{eq:gn}, we have 
$$
\left\|Iu^{\lambda}(\delta)\right\|_{H^1_{\lambda}}=o(1).
$$
We can iterate this process up to $N^{1/10-}$ times before doubling $\|Iu^{\lambda}(\delta)\|_{H^1_{\lambda}}$.
This process extends the local solution obtained by Proposition \ref{prop:I-local} to the time $O(N^{1/10-}\delta)$.
We choose $N$ such that by \eqref{eq:choice}
$$
N^{1/10-}\delta >\lambda^3T\sim N^{3(1-s)/(1+s)}T,
$$
which may be done for $s>29/31$, hence complete the claim of Theorem \ref{thm:main}. 

Furthermore, replacing $\lambda \sim N^{(1-s)/(1+s)}$ and $N\sim T^{10(1+s)/(31s-29)+}$ in \eqref{eq:lambda-s}, the growth of Sobolev norm of the solution $u(t)$ to the Cauchy problem \eqref{ZK} has a bound at least
$$
\|u(T)\|_{H^s} \lesssim \lambda^{1+s}\left\|Iu^{\lambda}(\lambda^3 T)\right\|_{H^1_{\lambda}}\lesssim T^{10(1-s)/(31s-29)+}.
$$
This finishes the proof. 
\end{proof}


\Addresses
\end{document}